\documentclass[11pt,reqno]{amsart}
\usepackage{amsmath}
\usepackage{amssymb}
\usepackage{microtype}
\usepackage[hidelinks]{hyperref}

\allowdisplaybreaks
\newtheorem{theorem}{Theorem}[section]
\newtheorem{proposition}[theorem]{Proposition}
\newtheorem{lemma}[theorem]{Lemma}
\newtheorem{corollary}[theorem]{Corollary}
\theoremstyle{definition}
\newtheorem{remark}[theorem]{Remark}
\numberwithin{equation}{section}
\newcommand{\R}{\mathbb R}
\newcommand{\N}{\mathbb N}
\newcommand{\dimH}{\dim_{\mathrm H}}
\newcommand{\dimC}{\dim_{\mathrm C}}
\newcommand{\spt}{\operatorname{spt}}
\newcommand{\id}{\operatorname{id}}
\newcommand{\diam}{\operatorname{diam}}
\newcommand{\esssup}{\operatorname*{ess\,sup}}
\newcommand{\E}{\mathbb E}

\begin{document}
\title[Conformal Dimension of Measures]{Conformal Dimension of Measures and Quasisymmetric Dimension Reduction}

\author{Hua Qiu}
\address{School of Mathematics, Nanjing University, Nanjing, 210093, P. R. China.}
\thanks{The research of Qiu was supported by the National Natural Science Foundation of China, grants 12471087 and 12531004. The research of Wang was supported by the Nanjing University PhD Student Zhujian Program, grant ZJJH2026B08.}
\email{huaqiu@nju.edu.cn}

\author{Qi Wang}
\address{School of Mathematics, Nanjing University, Nanjing, 210093, P. R. China.}
\email{602023210013@smail.nju.edu.cn}

\date{}

\begin{abstract}
We prove that the conformal dimension of every locally finite Borel measure
is either zero or infinite. The main ingredient is a quasisymmetric
dimension-reduction theorem: every full-support probability measure of
finite Hausdorff dimension on a separable metric space admits
quasisymmetrically equivalent metrics in which its Hausdorff dimension is
arbitrarily small. In particular, every locally finite Borel measure on a
doubling metric space has conformal dimension zero. We also prove that for
every $n\ge1$ and $p>0$ there are a metric space $X$, a quasisymmetric
homeomorphism $f:[0,1]^n\to X$, and a Borel set $E\subset[0,1]^n$ such that
$\dimH f(E)\le p$ and $\dimH([0,1]^n\setminus E)\le n-1+p$. The second bound
is sharp up to $p$: if $\dimH f(E)<1$, then
$\dimH([0,1]^n\setminus E)\ge n-1$.
\end{abstract}
\maketitle
\enlargethispage{3pt}

\section{Introduction}

\subsection{Background and main results}

We begin with notation. Let $(X,d)$ be a metric space. For $x\in X$ and
$r>0$, let $B_d(x,r)=\{y\in X:d(x,y)<r\}$ be the open ball centered at
$x$ with radius $r$. For $E\subset X$, let
$\diam_d E=\sup\{d(x,y):x,y\in E\}$ be its diameter. For $s\ge0$, the
\emph{$s$-dimensional Hausdorff measure} is
\[
 \mathcal H_d^s(E)
 =\lim_{\delta\downarrow0}
 \inf\left\{\sum_i(\diam_d U_i)^s:
 E\subset\bigcup_iU_i,\ \diam_d U_i\le\delta\right\},
\]
where the infimum is over countable covers.  The \emph{Hausdorff dimension}
of $E$ is
\[
 \dimH^d E=\inf\{s\ge0:\mathcal H_d^s(E)=0\}.
\]
We omit the metric subscript or superscript when it is clear.

Let $(Y,d_Y)$ be another metric space.  A homeomorphism $f:X\to Y$
is called \emph{quasisymmetric} if there is a homeomorphism
$\eta:[0,\infty)\to[0,\infty)$ such that
\[
 \frac{d_Y(f(x),f(a))}{d_Y(f(x),f(b))}
 \le \eta\!\left(\frac{d(x,a)}{d(x,b)}\right)
\]
for all distinct $x,a,b\in X$. A metric $D$ on $X$ is said to be
\emph{quasisymmetrically equivalent} to $d$ if the identity map
$\operatorname{id}_X:(X,d)\to(X,D)$ is quasisymmetric. The two concepts
are equivalent: quasisymmetric maps may be viewed as quasisymmetric changes
of metric on their domains. Tukia and V\"ais\"al\"a \cite{TV80} formulated quasisymmetry for
maps between general metric spaces, building on a one-dimensional form
implicit in Beurling and Ahlfors's study of boundary correspondences for
planar quasiconformal maps \cite{BA56}.

A fundamental quasisymmetric invariant is the \emph{conformal dimension} of a
metric space, introduced by Pansu \cite{P89} and defined by
\[
 \dimC X=\inf_f\dimH f(X),
\]
where the infimum is over all quasisymmetric homeomorphisms from $X$ onto
metric spaces.  This invariant measures how far the Hausdorff dimension of
the entire space can be reduced without changing its quasisymmetric geometry.

The space $X$ is called \emph{minimal for conformal dimension} if
$\dimC X=\dimH X$.
A classical theorem of Bishop and Tyson \cite{BT01} states that
$[0,1]\times E$ is minimal whenever $E\subset\R^{n-1}$ is compact. Questions
concerning the conformal dimension and conformal minimality of sets have been
studied in many settings; see, for example,
\cite{Kov06,Hak10,Mac11,CP14,Orp17,BHL25,MT26,MTTree26}.

The conformal dimension of a space concerns the geometry of the whole set and
does not distinguish between different measures on the same support. To study
the geometry seen by a measure rather than that of its entire support, Bate
and Orponen \cite{BO18} introduced the conformal dimension of a measure.

A Borel measure $\mu$ on $X$ is \emph{locally finite} if every point of
$X$ has an open neighborhood of finite $\mu$-measure.
For a Borel measure $\mu$ on $(X,d)$, its \emph{Hausdorff dimension} is
\[
 \dimH^d\mu
 =\inf\{\dimH^d E:E\subset X\text{ is Borel and }\mu(X\setminus E)=0\}.
\]
The \emph{support} of $\mu$, denoted by $S=\spt\mu$, is the set of points every
open neighborhood of which has positive $\mu$-measure. We say that $\mu$ has
\emph{full support} if $S=X$.

If $f:S\to Y$ is Borel, the \emph{push-forward measure}
$f_\#\mu$ on $Y$ is defined by $f_\#\mu(A)=\mu(f^{-1}(A))$.  The
\emph{conformal dimension of $\mu$} is
\begin{equation}\label{eq:conformal-dimension}
 \dimC\mu=\inf_f\dimH(f_\#\mu),
\end{equation}
where $f$ ranges over all quasisymmetric homeomorphisms from $(S,d)$ onto
metric spaces.

Throughout the paper, every measure is assumed to be nonzero, locally finite and to satisfy
$\mu(X\setminus S)=0$.
The latter assumption is needed because the quasisymmetric maps in the
definition act only on $S$. If $\mu(X\setminus S)>0$, then $f_\#\mu$ is
precisely the push-forward of the restricted measure $\mu|_S$; in particular,
$f_\#\mu(Y)=\mu(S)$. Thus the positive mass outside $S$ has no image, and the
resulting conformal dimension describes $\mu|_S$ rather than the original
measure $\mu$. Note that when $X$ is separable, every Borel measure satisfies
$\mu(X\setminus S)=0$.
Indeed, $X\setminus S$ is covered by open $\mu$-null sets, and a countable base
reduces this cover to a countable subcover. The same conclusion holds when
$X$ is doubling, since every doubling metric space is separable.

Bate and Orponen showed that the direct measure-theoretic analogue of the
Bishop--Tyson phenomenon can fail. For every $n\ge2$ and $0<s<n-1$, they
constructed a non-Ahlfors-regular compact set $E\subset\mathbb R^{n-1}$ such
that the restriction of $\mathcal H^{1+s}$ to $[0,1]\times E$ has conformal
dimension zero \cite{BO18}.
Thus the conformal geometry of a support and the conformal
geometry of the mass it carries can behave very differently.  This led them
to pose the following two questions \cite[Questions~1 and~2]{BO18}.

\vspace{0.2cm}
\noindent\textit{\textbf{Question $1$.} Do there exist measures with positive
and finite conformal dimension?}
\vspace{0.2cm}

\noindent\textit{\textbf{Question $2$.} Let $C\subset\mathbb R$ be the
middle-thirds Cantor set of dimension $s=\log 2/\log 3$, and let $\mu$ be the
$(1+s)$-dimensional Hausdorff measure on $[0,1]\times C$.  Does $\mu$ have
conformal dimension $1+s$?}
\vspace{0.2cm}

Our first theorem gives a complete answer to Question~1.

\begin{theorem}\label{thm:main}
Let $(X,d)$ be a metric space and let $\mu$ be a locally finite Borel measure
on $X$. Then $\dimC\mu\in\{0,+\infty\}$.
\end{theorem}

Theorem~\ref{thm:main} therefore implies that every locally finite Borel
measure on a doubling metric space has conformal dimension zero, so
Question~2 has a negative answer.
A metric space $(X,d)$ is \emph{doubling} if every ball can be covered by a
uniformly bounded number of balls of half the radius.

\begin{corollary}\label{cor:doubling}
Let $(X,d)$ be a doubling metric space and let $\mu$ be a locally finite Borel
measure on $X$. Then $\dimC\mu=0$.
\end{corollary}

\begin{proof}
By Assouad's embedding theorem \cite[Theorem~12.2]{H01}, for some
$0<\theta<1$ the snowflake $(X,d^\theta)$ admits a
bi-Lipschitz embedding into $\mathbb R^m$ for some $m\ge1$.
Thus there are an injective map $F:X\to\mathbb R^m$ and a constant
$C\ge1$ such that
$C^{-1}d(x,y)^\theta\le |F(x)-F(y)|\le C d(x,y)^\theta$ for all $x,y\in X$.
Hence
$\dimH^dX=\theta\dimH^{d^\theta}X\le\theta m<\infty$. Writing
$S=\spt\mu$, we obtain
$\dimC\mu\le\dimH^d\mu\le\dimH^dS\le\dimH^dX<\infty$, so
Theorem~\ref{thm:main} gives $\dimC\mu=0$.
\end{proof}

Corollary~\ref{cor:doubling} also has a short direct proof, independent of
Theorem~\ref{thm:main}. We collect it in Section~\ref{sec:doubling}.

Corollary~\ref{cor:doubling} naturally raises the question of what remains
true when the ambient space is not doubling. The following two examples show
that, without doubling, both values in Theorem~\ref{thm:main} can occur.

\medskip
\noindent\textit{Example 1: $\dimC\mu=\infty$.}
Let $X$ be an uncountable set with the discrete metric
\[
 d(x,y)=
 \begin{cases}
  0,&x=y,\\
  1,&x\ne y,
\end{cases}
\]
and let $\mu$ be the counting measure on $X$. Then $\mu$ is locally finite,
$\spt\mu=X$, and $X$ is neither separable nor doubling.  If
$h:X\to h(X)$ is a quasisymmetric homeomorphism, then $h_\#\mu$ is the counting
measure on $h(X)$, so its only full-measure Borel set is $h(X)$ itself. If
$h(X)$ had finite Hausdorff dimension, it would be separable, contradicting
the fact that $h$ is a homeomorphism and $X$ is nonseparable. Therefore
$\dimH(h_\#\mu)=\infty$ for every $h$, and $\dimC\mu=\infty$.

\medskip
\noindent\textit{Example 2: $\dimC\mu=0$.}
Let
\[
 X=\prod_{k=1}^{\infty}\{0,\ldots,k\},\qquad
 r_0=1,\quad r_m=\frac{1}{(m+1)!},
\]
and equip $X$ with the ultrametric $d(x,x)=0$ and $d(x,y)=r_m$ for
$x\ne y$, where $m$ is the largest integer such that
$x_1=y_1,\ldots,x_m=y_m$. Let $\mu$ be the product probability measure for
which the $k$-th coordinate is uniformly distributed on
$\{0,\ldots,k\}$. The metric $d$ induces the product topology, so $X$ is
compact, and $\mu$ has full support.

The space $X$ is not doubling. Indeed, every cylinder determined by the first
$m-1$ coordinates contains $m+1$ disjoint cylinders determined by the first
$m$ coordinates, so the number of half-radius balls required to cover a ball
is not uniformly bounded. Each cylinder of length $m$ has diameter $r_m$ and
$\mu$-measure $r_m$. For every $p>0$, the function $d^{1/p}$ is again an
ultrametric that is quasisymmetrically equivalent to $d$. With respect to
$d^{1/p}$, a cylinder of length $m$ has
diameter $r_m^{1/p}$ and measure $r_m$. This gives
$\dimH^{d^{1/p}}(\id_\#\mu)\le p$.
Letting $p\downarrow0$ gives $\dimC\mu=0$.

\medskip
The proof of Theorem~\ref{thm:main} is based on the following
dimension-reduction theorem.

\begin{theorem}\label{thm:reduction}
Let $(X,d)$ be a separable metric space and let $\mu$ be a Borel probability
measure with full support on $X$. If $\dimH^d\mu<\infty$, then for every
$p>0$ there is a metric $D$ on $X$, quasisymmetrically equivalent to $d$,
such that $\dimH^D\mu\le p$. Consequently, $\dimC\mu=0$.
\end{theorem}

The proof of Theorem~\ref{thm:reduction} adapts the weighted hyperbolic-filling
construction of Miller and Tian \cite{MT26,MTTree26} and combines it with a
probabilistic choice of weights. In a hyperbolic filling, vertices represent points of $X$ at
successively finer scales, horizontal edges join nearby points at the same
scale, and vertical edges connect consecutive scales. We choose the weights
so that small neighborhoods are inexpensive at most scales while every chain
crossing the filling retains a uniform positive cost. The latter bound
prevents the metric from collapsing and, together with control of ball
diameters, yields quasisymmetry. Our setting introduces two additional
features: $X$ need not be compact, so the filling may have countably infinite
vertex sets and need not be locally finite; and the metric is defined directly
on $X$ from the costs of finite chains, rather than on the boundary of a
reweighted graph. Our later application uses the diameter estimate only at
$\mu$-almost every point, with the weights chosen by random Borel partitions
adapted to $\mu$. These partitions ensure that almost every point lies near a
partition boundary at only a controlled proportion of scales, which gives
$\dimH^D\mu\le p$.

\medskip

Write $\mathcal L^n$ for $n$-dimensional Lebesgue measure. Note that
Corollary~\ref{cor:doubling} immediately implies that $\mathcal L^n$ has
conformal dimension zero for every $n\ge1$. This is closely related to earlier
constructions of singular quasisymmetric maps. Research in this direction
began with Tukia's one-dimensional result \cite{T89}: for every $p>0$ there
are a quasisymmetric homeomorphism $f:[0,1]\to[0,1]$ and a set
$E\subset[0,1]$ such that
\begin{equation}\label{eq:tukia}
 \dimH f(E)\le p,
 \qquad
 \dimH([0,1]\setminus E)\le p.
\end{equation}
Romney \cite{R19} later proved that for every $n\ge2$ and $p>0$ there
are a Borel set $E\subset[0,1]^n$, a metric space $X$, and a quasisymmetric
homeomorphism $f:[0,1]^n\to X$ such that
\begin{equation}\label{eq:romney}
 \dimH f(E)\le p,
 \qquad
 \mathcal L^n([0,1]^n\setminus E)=0.
\end{equation}
Romney left open the following question \cite[Section~1]{R19}.

\vspace{0.2cm}
\noindent\textit{\textbf{Question $3$.} To what extent can the Hausdorff
dimension of $[0,1]^n\setminus E$ be reduced?}
\vspace{0.2cm}

We can give a sharp quantitative answer by controlling simultaneously the
dimension of the image of a set and the dimension of its complement.

\begin{theorem}\label{thm:singular}
Let $n\geq1$. For every $p>0$, there exist a metric space $X$, a
quasisymmetric homeomorphism $f:[0,1]^n\to X$, and a Borel set
$E\subset[0,1]^n$ such that
\[
 \dimH f(E)\le p
 \quad\text{and}\quad
 \dimH([0,1]^n\setminus E)\le n-1+p.
\]
\end{theorem}

\begin{remark}
When $n=1$, the conclusion is Tukia's theorem~\eqref{eq:tukia}.  When $n\ge2$,
it refines Romney's theorem~\eqref{eq:romney} by replacing the Lebesgue-null
conclusion with an explicit Hausdorff-dimension bound. This bound is sharp
up to the additive term $p$ when $0<p<1$. Indeed, we can further show that
(see Proposition~\ref{prop:sharpness}) whenever a homeomorphism
$f:[0,1]^n\to X$ and a set $E\subset[0,1]^n$
satisfy $\dimH f(E)<1$, one necessarily has
\[
 \dimH([0,1]^n\setminus E)\ge n-1.
\]
\end{remark}

\subsection{Theorem~\ref{thm:reduction} implies Theorem~\ref{thm:main}}
\label{subsec:intro-proof}

We now deduce Theorem~\ref{thm:main} from the dimension-reduction theorem,
Theorem~\ref{thm:reduction}.

\begin{proof}[Proof of Theorem~\ref{thm:main}]
Let $S=\spt\mu$.  Suppose first that $S$ is nonseparable, and let
$f:S\to Y$ be any quasisymmetric homeomorphism onto a metric space $Y$.
The measure $f_\#\mu$ has full support on $Y$, so every Borel set of full
$f_\#\mu$-measure is dense in $Y$. If $\dimH(f_\#\mu)<\infty$, there is a
full-measure Borel set $A\subset Y$ with $\dimH A<\infty$. Every metric space
of finite Hausdorff dimension is separable, so $A$, and hence its closure $Y$,
would be separable. Since $f$ is a homeomorphism, this would make $S$
separable, a contradiction. Therefore $\dimC\mu=+\infty$ when $S$ is
nonseparable.

Assume now that $S$ is separable.  Local finiteness gives every point of $S$
an open neighborhood of finite measure.  A countable base yields a countable
subcover, so $S$ is the union of countably many Borel sets of finite measure.
Thus $\mu$ is $\sigma$-finite on $S$. Choose pairwise disjoint Borel
sets $F_j$ of finite measure whose union is $S$. Define the Borel function
\[
 w=\sum_{j\ge1}\frac{2^{-j}}{1+\mu(F_j)}\mathbf 1_{F_j}.
\]
Then $w>0$ on $S$ and $0<\int_Sw\,d\mu<\infty$.  Define a probability
measure $\nu$ on $S$ by
\[
 d\nu=\frac{w\,d\mu}{\int_Sw\,d\mu}.
\]
Because $w$ is strictly positive, $\mu$ and $\nu$ have exactly the same null
sets; in particular, $\nu$ has full support.  Consequently, for every
quasisymmetric homeomorphism $f:S\to Y$,
\begin{equation}\label{eq:intro-equivalent-measures}
 \dimH(f_\#\mu)=\dimH(f_\#\nu)
 \quad\text{and}\quad
 \dimC\mu=\dimC\nu.
\end{equation}

Suppose that $\dimC\mu<\infty$.  By
\eqref{eq:conformal-dimension} and
\eqref{eq:intro-equivalent-measures}, there is a quasisymmetric
homeomorphism $f:S\to Y$ such that $\dimH(f_\#\nu)<\infty$.  The space $Y$
is separable, and $f_\#\nu$ is a probability measure with full support on $Y$.
For every $p>0$, Theorem~\ref{thm:reduction} supplies a metric $D_p$ on $Y$
such that $\id:(Y,d_Y)\to(Y,D_p)$ is quasisymmetric and
$\dimH^{D_p}(f_\#\nu)\le p$.  Composing this identity map with $f$ and using
\eqref{eq:intro-equivalent-measures} gives $\dimC\mu\le p$. Since $p>0$ is
arbitrary, $\dimC\mu=0$.
\end{proof}

\medskip
We organize the paper as follows. Section~\ref{sec:filling}
develops the weighted hyperbolic filling and the associated chain metric.
Section~\ref{sec:reduction-proof} uses random Borel partitions to prove
Theorem~\ref{thm:reduction}. Section~\ref{sec:doubling} gives a direct proof
of Corollary~\ref{cor:doubling} that does not use Theorem~\ref{thm:main}.
Section~\ref{sec:singular} proves Theorem~\ref{thm:singular} and the sharpness
statement in Proposition~\ref{prop:sharpness}.

Before ending this section, we note that the full-measure definition of
$\dimH^d\mu$ used in this paper is often called the \emph{upper Hausdorff
dimension} of $\mu$. One may instead use the \emph{lower Hausdorff dimension},
defined by
$\inf\{\dimH^d E:E\subset X\text{ is Borel and }\mu(E)>0\}$, and define the
corresponding conformal dimension. We do not know whether
Theorem~\ref{thm:main} remains valid when conformal dimension is defined using
the lower Hausdorff dimension. However, the analogues of
Corollary~\ref{cor:doubling} and Theorem~\ref{thm:reduction} also hold when
the lower Hausdorff dimension is used throughout.

\section{Hyperbolic fillings}\label{sec:filling}

This section carries out the hyperbolic-filling construction used in the
proof of Theorem~\ref{thm:reduction}.
Here a \emph{hyperbolic filling} means a rooted
graph organized into levels.  Its vertices represent points of $X$ at
successively finer scales.  Horizontal edges join nearby points on the same
level, and vertical edges join points on consecutive levels. Such multiscale graphs are standard tools
in metric geometry; see, for example,
\cite{BP03,BS07,CP13,BS18,BBS22,ESS25}.

We carry out the construction in three steps. Subsection~\ref{subsec:nets} defines the filling and its vertex
weights, Subsection~\ref{subsec:chain-bound} proves a uniform lower bound for
path costs, and Subsection~\ref{subsec:chain-metric} transfers the weights to
$X$, defines the metric, and proves the required diameter estimate.

\subsection{Separated sets and admissible weights}\label{subsec:nets}

This subsection constructs the vertices and edges of the filling and then
assigns compatible weights to its vertices.

Let $(X,d)$ be bounded, complete, and separable, with $\diam X<1$. Set
$\alpha_0=10^{-3}$ and fix
\[
0<\alpha<\alpha_0.
\]
For $r>0$, a set is \emph{$r$-separated} if distinct points in it are at
distance at least $r$, and it is maximal $r$-separated if it is not properly
contained in another $r$-separated set.  Choose a base point $o\in X$ and set
$A_0=\{o\}$.  Inductively choose nested
maximal $\alpha^n$-separated sets
$A_n$; thus
\[
A_0\subset A_1\subset A_2\subset\cdots,
\qquad
X\subset\bigcup_{x\in A_n}B(x,\alpha^n),
\]
and $d(x,y)\ge\alpha^n$ for all distinct $x,y\in A_n$.  Inductively extend
$A_{n-1}$ to a maximal $\alpha^n$-separated set $A_n$ by Zorn's lemma.
Each $A_n$ is countable because $X$ is separable.
Put $V_n=A_n\times\{n\}$ and $V=\bigcup_{n\ge0}V_n$.

Two vertices $u=(x,n)$ and $v=(y,n)$ at the same level are
\emph{horizontal neighbors}, written $u\sim v$, if
$d(x,y)<8\alpha^n$; we also declare $u\sim u$. Every
$v=(y,n+1)$ is assigned a \emph{parent} $g(v)=(x,n)$ satisfying
$d(x,y)<\alpha^n$. Such a parent need not be unique; we simply choose and
fix one for each vertex. The unordered pair $\{v,g(v)\}$ is called a
\emph{vertical edge}. The vertical edges form a tree rooted at $(o,0)$:
every vertex has a unique finite parent chain ending at $(o,0)$. Throughout,
$\sim$ refers only to horizontal adjacency. For vertices $u=(x,n)$ and
$v=(y,k)$, possibly at different levels, we use the convention
$d(u,v)=d(x,y)$.

The choice $\alpha<\alpha_0$ gives the following useful observation.  If
$n\ge1$ and
\[
u_0\sim u_1\sim\cdots\sim u_N,\qquad u_i\in V_n,\qquad N\le100,
\]
then
\begin{equation}\label{eq:parent-short-chain}
	d(g(u_0),g(u_N))
	<2\alpha^{n-1}+8N\alpha^n
	\le(2+800\alpha)\alpha^{n-1}
	<8\alpha^{n-1}.
\end{equation}
Thus $g(u_0)\sim g(u_N)$. In particular, the parents of horizontal
neighbors are horizontal neighbors. 

An assignment $\sigma:V\setminus V_0\to[0,\infty)$ is called
\emph{admissible} if it has the following property. Whenever
$u_0=(x_0,n)\sim\cdots\sim u_N=(x_N,n)$ and $v=(y,n-1)$ satisfy
\begin{equation}\label{eq:annular-endpoints}
	B(y,\alpha^{n-1})\cap B(x_0,4\alpha^n)\ne\varnothing,
	\qquad
	(X\setminus B(y,2\alpha^{n-1}))
	\cap B(x_N,4\alpha^n)\ne\varnothing,
\end{equation}
one has
\begin{equation}\label{eq:admissible}
	\sum_{i=0}^N\sigma(u_i)\ge1.
\end{equation}

Fix $0<\eta<1/2$. For $u\in V_n$, $n\ge1$, define
\begin{equation}\label{eq:mweight}
	a(u)=\max\!\left\{\eta,
	\min\!\left\{1-\eta,
	2\sup_{u\sim v\sim w}\sigma(w)\right\}\right\}.
\end{equation}
Thus $\eta\le a(u)\le1-\eta$.

We next construct a positive number $\pi(u)$ at every vertex. Set
$\pi((o,0))=1$. Suppose that $\pi$ has been defined on $V_{n-1}$ and
satisfies
\begin{equation}\label{eq:horizontal-pi}
	\eta\le\frac{\pi(u)}{\pi(v)}\le\eta^{-1}
	\quad\text{whenever }u,v\in V_{n-1}\text{ and }u\sim v.
\end{equation}
For $u\in V_n$, put $q(u)=\pi(g(u))a(u)$. Write $v\succ u$ if
$v\sim u$ and $q(v)>\eta^{-1}q(u)$. There is no chain
$v\succ u\succ w$. Indeed, $v\sim u\sim w$, so
\eqref{eq:parent-short-chain} gives $g(v)\sim g(w)$,
whereas such a chain would give
\[
\pi(g(v))>q(v)>\eta^{-2}q(w)
\ge\eta^{-1}\pi(g(w)),
\]
contrary to \eqref{eq:horizontal-pi}. Define
\begin{equation}\label{eq:pi-recursion}
	\pi(u)=
	\begin{cases}
		q(u),&\text{if no }v\succ u,\\
		\eta\sup_{v\succ u}q(v),&\text{otherwise.}
	\end{cases}
\end{equation}

The next lemma adapts \cite[Lemmas~5.2--5.3]{MT26} to the present setting,
where the sets $A_n$ may be infinite. In \cite{MT26}, the corresponding
sets are finite, so the relevant suprema are attained; the argument here
does not require a maximizing vertex.

\begin{lemma}\label{lem:recursive}
	For $u,v\in V_n$,
	\begin{align}
		q(u)&\le\pi(u)\le\sup_{w\sim u}q(w),
		\label{eq:pi-q}\\
		u\sim v&\quad\Longrightarrow\quad
		\eta\le\frac{\pi(u)}{\pi(v)}\le\eta^{-1}.
		\label{eq:pi-horizontal-new}
	\end{align}
	If
	\begin{equation}\label{eq:rho-def}
		\rho(u)=\frac{\pi(u)}{\pi(g(u))},
	\end{equation}
	then
	\begin{equation}\label{eq:rho-bounds}
		\eta\le a(u)\le\rho(u)
		\le\sup_{w\sim u}a(w)\le1-\eta.
	\end{equation}
\end{lemma}

\begin{proof}
	The first inequality in \eqref{eq:pi-q} follows from the definition. For
	the second inequality in \eqref{eq:pi-q}, if the second line of
	\eqref{eq:pi-recursion} applies, then
	$\pi(u)=\eta\sup_{v\succ u}q(v)\le\sup_{v\sim u}q(v)$; the other case is
	immediate.
	
	Let $u\sim v$. By symmetry, it is enough to show that
	$\pi(v)\ge\eta\pi(u)$. Suppose first that $\pi(u)=q(u)$. If $u\succ v$,
	then
	\[
	\pi(v)=\eta\sup_{w\succ v}q(w)\ge\eta q(u)=\eta\pi(u).
	\]
	If $u\not\succ v$, then $q(u)\le\eta^{-1}q(v)$, and hence
	\[
	\pi(v)\ge q(v)\ge\eta q(u)=\eta\pi(u).
	\]
	
	Suppose now that $\pi(u)>q(u)$. Then
	$\pi(u)=\eta\sup_{w\succ u}q(w)$. For every $w\succ u$, the relation
	$w\sim u\sim v$ and \eqref{eq:parent-short-chain} give
	$g(w)\sim g(v)$.
	Using \eqref{eq:horizontal-pi} and $\eta\le a(w)\le1-\eta$, we obtain
	\[
	q(w)=\pi(g(w))a(w)
	\le\pi(g(w))
	\le\eta^{-1}\pi(g(v))
	\le\eta^{-2}q(v).
	\]
	Taking the supremum over $w\succ u$ gives
	\[
	\pi(u)=\eta\sup_{w\succ u}q(w)
	\le\eta^{-1}q(v)
	\le\eta^{-1}\pi(v).
	\]
	This proves \eqref{eq:pi-horizontal-new}.
	
	Since $q(u)=\pi(g(u))a(u)$, the lower bounds in
	\eqref{eq:rho-bounds} follow from \eqref{eq:pi-q}. If $\pi(u)=q(u)$,
	the upper bounds are immediate. Otherwise, for $v\succ u$ we have
	$g(v)\sim g(u)$, and therefore
	\[
	\eta q(v)=\eta\pi(g(v))a(v)
	\le\pi(g(u))a(v).
	\]
	Taking the supremum proves the upper bounds in \eqref{eq:rho-bounds}.
\end{proof}

Lemma~\ref{lem:recursive} shows that the recursive construction can be
continued at every level and yields the following product formula.

For $u\in V_n$ and $0\le j\le n$, let $g(u)_j\in V_j$ denote the unique
level-$j$ vertex on the parent chain from $(o,0)$ to $u$. We call it the
\emph{level-$j$ ancestor} of $u$. Thus
$g(g(u)_{j+1})=g(u)_j$, $g(u)_0=(o,0)$, and $g(u)_n=u$. By
\eqref{eq:rho-def} and \eqref{eq:rho-bounds},
\begin{equation}\label{eq:pi-product}
	\pi(u)=\prod_{j=1}^n\rho(g(u)_j),
	\qquad
	(1-\eta)^{-1}\le\frac{\pi(g(u))}{\pi(u)}\le\eta^{-1}.
\end{equation}

\subsection{The chain lower bound}\label{subsec:chain-bound}

We first prove that every horizontal
chain crossing the annulus between radii $\alpha^{n-1}$ and
$2\alpha^{n-1}$ has total $\rho$-cost at least $1$.

\begin{lemma}\label{lem:annular}
	Let $v\in V_{n-1}$ and let $u_0\sim\cdots\sim u_N$ be a level-$n$
	chain satisfying \eqref{eq:annular-endpoints}. Then
	\begin{equation}\label{eq:annular-rho}
		\sum_{i=1}^N
		\inf\{\rho(w):w\sim u_{i-1}\text{ or }w\sim u_i\}\ge1.
	\end{equation}
\end{lemma}

\begin{proof}
	Writing $u_i=(x_i,n)$ and $v=(y,n-1)$, the two conditions in
	\eqref{eq:annular-endpoints} give points
	$\xi_0\in B(y,\alpha^{n-1})\cap B(x_0,4\alpha^n)$ and
	$\xi_N\in (X\setminus B(y,2\alpha^{n-1}))\cap B(x_N,4\alpha^n)$. Thus
	\[
	d(x_0,x_N)\ge d(y,\xi_N)-d(y,\xi_0)-d(\xi_0,x_0)-d(\xi_N,x_N)>(1-8\alpha)\alpha^{n-1}>8\alpha^n,
	\]
	where the last inequality follows from $\alpha<10^{-3}$. Hence the chain
	has at least two edges. By \eqref{eq:rho-bounds}, it is enough to prove
	\begin{equation}\label{eq:annular-a}
		\sum_{i=1}^N
		\inf\{a(w):w\sim u_{i-1}\text{ or }w\sim u_i\}\ge1.
	\end{equation}
	
	If one of the infima in \eqref{eq:annular-a} equals $1-\eta$, then that
	term together with any other term contributes at least
	$(1-\eta)+\eta=1$. Suppose, therefore, that every infimum is strictly
	less than $1-\eta$. Fix $i$ and set
	\[
	s_i=\sigma(u_{i-1})+\sigma(u_i).
	\]
	For every $w$ horizontally adjacent to $u_{i-1}$ or $u_i$, both
	$u_{i-1}$ and $u_i$ lie in the two-step horizontal neighborhood of $w$.
	Consequently,
	\[
	2\sup_{w\sim z\sim z'}\sigma(z')
	\ge2\max\{\sigma(u_{i-1}),\sigma(u_i)\}
	\ge s_i.
	\]
	Thus $a(w)\ge\min\{1-\eta,s_i\}$. The assumption on the corresponding
	infimum forces $s_i<1-\eta$, and hence that infimum is at least $s_i$.
	Summing in $i$ gives
	\[
	\begin{aligned}
		\sum_{i=1}^N
		\inf\{a(w):w\sim u_{i-1}\text{ or }w\sim u_i\}
		&\ge\sum_{i=1}^Ns_i\\
		&=\sigma(u_0)+2\sum_{i=1}^{N-1}\sigma(u_i)+\sigma(u_N)\\
		&\ge\sum_{i=0}^N\sigma(u_i)\ge1,
	\end{aligned}
	\]
	where the last inequality is \eqref{eq:admissible}.
	This proves \eqref{eq:annular-rho}.
\end{proof}

For vertices $u\in V_r$ and $v\in V_s$, let $k(u,v)$ be the largest
integer $0\le j\le\min\{r,s\}$ for which their level-$j$ ancestors are
horizontal neighbors, and put
\begin{equation}\label{eq:Piuv}
	\Pi(u,v)=\max\{\pi(g(u)_{k(u,v)}),
	\pi(g(v)_{k(u,v)})\}.
\end{equation}

The next lemma, whose proof follows \cite[Lemma~5.4]{MT26}, bounds the
ancestor weight $\Pi$ from \eqref{eq:Piuv} by a constant multiple of every
path cost.  This estimate
is used to construct the chain metric in
Subsection~\ref{subsec:chain-metric}.

\begin{lemma}\label{lem:discrete-chain}
	There exists $c_\eta>0$ with the following property. If
	$P=(u_0,\ldots,u_m)$ is a finite path consisting of horizontal and vertical
	edges, then
	\[
	\sum_{i=0}^m\pi(u_i)\ge c_\eta\Pi(u_0,u_m).
	\]
\end{lemma}

\begin{proof}
	For $u\in V_n$, set
	\[
	\pi_*(u)=\inf_{w\sim u}\pi(w).
	\]
	Since $u\sim u$, \eqref{eq:pi-horizontal-new} gives
	\begin{equation}\label{eq:pi-star}
		\eta\pi(u)\le\pi_*(u)\le\pi(u).
	\end{equation}
	For a horizontal edge $u\sim v$, assign the length
	\[
	\ell(uv)=\min\{\pi_*(u),\pi_*(v)\},
	\]
	and for a vertical edge $g(u)u$, assign the length
	\[
	\ell(g(u)u)=\eta^{-4}\pi_*(u).
	\]
	For a finite path $R=(v_0,\ldots,v_M)$, write
	$\ell(R)=\sum_{i=1}^M\ell(v_{i-1}v_i)$.
	
	For each edge of $P$, choose an endpoint at the higher level if the edge
	is vertical, and either endpoint if it is horizontal.  For a horizontal
	edge, its length is at most the weight of either chosen endpoint by
	\eqref{eq:pi-star}; for a vertical edge $g(u)u$, its length is at most
	$\eta^{-4}\pi(u)$.  Each occurrence $u_i$ in the path can be chosen only
	for the edge entering it and the edge leaving it, so each occurrence is
	chosen at most twice.  Therefore
	\begin{equation}\label{eq:path-cost-compare}
		\ell(P)\le2\eta^{-4}\sum_{i=0}^m\pi(u_i),
		\qquad
		\sum_{i=0}^m\pi(u_i)\ge\frac{\eta^4}{2}\ell(P).
	\end{equation}
	
	Let $\Gamma(v)$ be the family of level-$n$ horizontal chains satisfying
	\eqref{eq:annular-endpoints} relative to $v\in V_{n-1}$.
	
	\medskip
	\noindent{\itshape Claim 1. If $R\in\Gamma(v)$, then
	$\ell(R)\ge\pi_*(v)$.\par}
	\medskip

	Indeed, consider the initial subchain ending at the first exit from the
	relevant annulus. For each edge $u_{i-1}\sim u_i$ of this subchain and
	each $w$ horizontally adjacent to $u_{i-1}$ or $u_i$, the parent $g(w)$
	is horizontally adjacent to $v$. Choose
	$u_j\in\{u_{i-1},u_i\}$ with $d(u_j,v)<2\alpha^{n-1}$, and choose
	$u_k\in\{u_{i-1},u_i\}$ with $w\sim u_k$. Then
	\[
	d(g(w),v)\le d(g(w),w)+d(w,u_k)+d(u_k,u_j)+d(u_j,v)<(3+16\alpha)\alpha^{n-1}<8\alpha^{n-1}.
	\]
	Hence
	\[
	\pi(w)=\pi(g(w))\rho(w)\ge\pi_*(v)\rho(w),
	\]
	and therefore
	\[
	\ell(u_{i-1}u_i)
	\ge\pi_*(v)
	\inf\{\rho(w):w\sim u_{i-1}\text{ or }w\sim u_i\}.
	\]
	Summing and applying Lemma~\ref{lem:annular} proves Claim~1.
	\nobreak\hspace{0.5em}$\square$\par\medskip
	\noindent{\itshape Claim 2. Suppose that a subpath
	$z_0,z_1,\ldots,z_M$ has endpoints in $V_{h-1}$
	and all its interior vertices in $V_h$. Then there is a horizontal path
	$R\subset V_{h-1}$ from $z_0$ to $z_M$ such that
	\begin{equation}\label{eq:e1}
		\ell(R)\le\ell(z_1,\ldots,z_M).
	\end{equation}
	\par}
	\medskip

	If $z_0\sim z_M$, take $R=(z_0,z_M)$. Since
	$g(z_{M-1})=z_M$, \eqref{eq:pi-star} and \eqref{eq:rho-bounds} give
	\[
	\ell(z_0z_M)
	\le\pi(z_M)
	\le\eta^{-1}\pi(z_{M-1})
	\le\eta^{-2}\pi_*(z_{M-1})
	\le\ell(z_{M-1}z_M).
	\]
	
	Assume $z_0\not\sim z_M$. Set $i_0=1$ and $b_0=z_0=g(z_1)$.
	Whenever $b_r=g(z_{i_r})\not\sim z_M$, choose the least
	$i_{r+1}>i_r$ such that
	\[
	(z_{i_r},\ldots,z_{i_{r+1}})\in\Gamma(b_r),
	\]
	and put $b_{r+1}=g(z_{i_{r+1}})$. This index exists and satisfies
	$i_{r+1}\le M-1$ because $g(z_{M-1})=z_M$. Its minimality gives
	$d(b_r,z_{i_{r+1}-1})<2\alpha^{h-1}$, and hence
	\[
		d(b_r,b_{r+1})\le d(b_r,z_{i_{r+1}-1})+d(z_{i_{r+1}-1},z_{i_{r+1}})+d(z_{i_{r+1}},b_{r+1})
		<(3+8\alpha)\alpha^{h-1}
		<8\alpha^{h-1}.
	\]
	Thus $b_r\sim b_{r+1}$, and Claim~1 yields
	\begin{equation}\label{eq:selected-subchain}
		\ell(b_rb_{r+1})
		\le\pi_*(b_r)
		\le\ell(z_{i_r},\ldots,z_{i_{r+1}}).
	\end{equation}
	The indices increase strictly, so after finitely many steps
	$b_L\sim z_M$. The selected subchains have disjoint edge sets, and the
	first case bounds the last edge $b_Lz_M$ by
	$\ell(z_{M-1}z_M)$. Thus $R=(b_0,\ldots,b_L,z_M)$ satisfies
	\eqref{eq:e1} by summing \eqref{eq:selected-subchain} over the disjoint
	subchains. Since $b_0=z_0$, this proves Claim~2.
	\nobreak\hspace{0.5em}$\square$\par\medskip
	
	For $n\ge0$, let $\mathcal G_{\le n}$ denote the subgraph induced by
	$\bigcup_{j=0}^nV_j$.
	
	\medskip
	\noindent{\itshape Claim 3. Suppose that $x\in V_r$, $y\in V_s$, and
	$r\le s$. Every finite path
	from $x$ to $y$ can be replaced, without increasing its $\ell$-length,
	by a path in $\mathcal G_{\le r}$ from $x$ to $g(y)_r$. Moreover, if
	$x,y\in V_n$, $x\not\sim y$, and $g(x)\not\sim g(y)$, then it can be
	replaced by a path in $\mathcal G_{\le n-1}$ from $g(x)$ to $g(y)$.\par}
	\medskip

	To prove the first assertion, let $H$ be the largest level visited by the
	path. If $H>s$, every maximal subchain of level-$H$ vertices lies in a
	complete block whose endpoints are in $V_{H-1}$. Apply Claim~2 to all
	such blocks. The resulting path does not meet $V_H$ and is no longer than
	the original path. Since the path is finite, iteration reduces it to
	$\mathcal G_{\le s}$.
	
	If $r=s$, this proves the first assertion. Suppose $r<s$. After lowering
	all complete level-$s$ blocks, the vertices of the path that remain in
	$V_s$, if any, form the terminal subchain
	\[
	z_0,z_1,\ldots,z_M=y,
	\qquad z_0\in V_{s-1},\quad z_1,\ldots,z_M\in V_s.
	\]
	Reverse this terminal subchain and temporarily add the parent $g(y)$ at
	its beginning:
	\[
	g(y),\ y=z_M,\ z_{M-1},\ldots,z_1,\ z_0.
	\]
	This is a complete level-$s$ block. By Claim~2, it
	can be replaced by a horizontal path $R\subset V_{s-1}$ from $g(y)$ to
	$z_0$ satisfying
	\[
	\ell(R)\le\ell(y,z_{M-1},\ldots,z_1,z_0).
	\]
	Reverse $R$ and append it to the part of the original path ending at
	$z_0$. This replaces the terminal level-$s$ subchain by a path ending at
	$g(y)$, without increasing length. Repeating at levels
	$s-1,\ldots,r+1$ gives a path in $\mathcal G_{\le r}$ from $x$ to
	$g(y)_r$.
	
	For the second assertion, first apply the first assertion so that
	the path lies in $\mathcal G_{\le n}$. If it meets $V_{n-1}$, let $a$ and
	$b$ be its first and last vertices in $V_{n-1}$. Apply the first assertion
	to the reverse of the initial segment from $x$ to $a$, to the middle
	segment from $a$ to $b$, and to the terminal segment from $b$ to $y$.
	After reversing the first replacement and concatenating, we obtain the
	required path from $g(x)$ to $g(y)$ in $\mathcal G_{\le n-1}$.
	
	It remains to treat the case in which the path never meets $V_{n-1}$.
	It is then a horizontal chain
	\[
	x=x_0\sim x_1\sim\cdots\sim x_M=y
	\]
	in $V_n$. Put $p_i=g(x_i)$. Starting with $j_0=0$, choose
	$j_{q+1}>j_q$ minimally so that
	\[
	(x_{j_q},\ldots,x_{j_{q+1}})\in\Gamma(p_{j_q}),
	\]
	and stop when the remaining chain contains no such crossing. Since
	$p_0\not\sim p_M$, at least one crossing is selected. The argument
	leading to \eqref{eq:selected-subchain} gives, for every selected crossing,
	\[
	p_{j_q}\sim p_{j_{q+1}},
	\qquad
	\ell(p_{j_q}p_{j_{q+1}})
	\le\ell(x_{j_q},\ldots,x_{j_{q+1}}).
	\]
	
	Let $0=j_0<\cdots<j_L$ be the selected indices. The minimality of the
	last crossing gives
	\[
	d(p_{j_{L-1}},x_{j_L-1})<2\alpha^{n-1},
	\]
	whereas the absence of a crossing after $j_L$ gives
	$d(p_{j_L},x_M)<2\alpha^{n-1}$. Therefore
	\[
	\begin{aligned}
		d(p_{j_{L-1}},p_M)
		&\le d(p_{j_{L-1}},x_{j_L-1})+d(x_{j_L-1},x_{j_L})
		+d(x_{j_L},p_{j_L})+d(p_{j_L},x_M)+d(x_M,p_M)\\
		&<2\alpha^{n-1}+8\alpha^n+\alpha^{n-1}
		+2\alpha^{n-1}+\alpha^{n-1}\\
		&=(6+8\alpha)\alpha^{n-1}
		<8\alpha^{n-1}.
	\end{aligned}
	\]
	Thus $p_{j_{L-1}}\sim p_M$. By Claim~1, the last selected subchain has
	length at least $\pi_*(p_{j_{L-1}})$, which is at least the length of the
	last horizontal edge:
	\[
	\ell(p_{j_{L-1}}p_M)
	\le\pi_*(p_{j_{L-1}})
	\le\ell(x_{j_{L-1}},\ldots,x_{j_L}).
	\]
	The selected subchains have disjoint edge sets, so
	\[
	p_0\sim p_{j_1}\sim\cdots\sim p_{j_{L-1}}\sim p_M
	\]
	is no longer than the original chain. This proves Claim~3.
	\nobreak\hspace{0.5em}$\square$\par\medskip
	
	We now finish the proof. The case $m=0$ is immediate, so assume $m\ge1$.
	After interchanging the endpoints if necessary, write
	$u_0\in V_r$, $u_m\in V_s$, with $r\le s$, and put
	$k=k(u_0,u_m)$. Claim~3 first gives a path no longer than $P$ from
	$u_0$ to $g(u_m)_r$ in $\mathcal G_{\le r}$.
	
	Suppose $k<r$. Repeatedly apply the second assertion of Claim~3 until the
	endpoints have reached level $k+1$. We obtain a path $R$ no longer than
	$P$, with endpoints
	\[
	x=g(u_0)_{k+1},\qquad y=g(u_m)_{k+1},
	\]
	such that $x\not\sim y$ and $g(x)\sim g(y)$. By
	Lemma~\ref{lem:recursive} and \eqref{eq:pi-star}, every edge incident to a
	vertex $z$ has length at least $\eta^2\pi(z)$. Indeed, if $zw$ is
	horizontal, then
	\[
	\pi_*(z)\ge\eta\pi(z),
	\qquad
	\pi_*(w)\ge\eta\pi(w)\ge\eta^2\pi(z),
	\]
	and therefore $\ell(zw)\ge\eta^2\pi(z)$. If $w=g(z)$, then
	\[
	\ell(zw)=\eta^{-4}\pi_*(z)
	\ge\eta^{-3}\pi(z)\ge\eta^2\pi(z).
	\]
	Finally, if $g(w)=z$, then
	\[
	\ell(zw)=\eta^{-4}\pi_*(w)
	\ge\eta^{-3}\pi(w)
	\ge\eta^{-2}\pi(z)
	\ge\eta^2\pi(z).
	\]
	Applying this estimate to the first edge of $R$ and using
	Lemma~\ref{lem:recursive} gives
	\[
	\begin{aligned}
		\ell(P)&\ge\ell(R)\ge\eta^2\pi(x)
		\ge\eta^3\pi(g(x))\\
		&\ge\eta^4\max\{\pi(g(x)),\pi(g(y))\}
		=\eta^4\Pi(u_0,u_m).
	\end{aligned}
	\]
	
	If $k=r$, applying the same estimate to the first edge of the original
	path and using $u_0\sim g(u_m)_r$ together with
	Lemma~\ref{lem:recursive} gives
	\[
	\ell(P)\ge\eta^2\pi(u_0)
	\ge\eta^3\max\{\pi(u_0),\pi(g(u_m)_r)\}
	=\eta^3\Pi(u_0,u_m).
	\]
	This is stronger than the bound required in the case $k<r$. Combining
	both cases with
	\eqref{eq:path-cost-compare}
	gives
	\[
	\sum_{i=0}^m\pi(u_i)
	\ge\frac{\eta^8}{2}\Pi(u_0,u_m).
	\]
	Thus the lemma holds with $c_\eta=\eta^8/2$.
\end{proof}

\subsection{The chain metric}\label{subsec:chain-metric}

We now convert the vertex weights into a scale function on $X$ and then take
the associated intrinsic chain metric.  Lemma~\ref{lem:discrete-chain}
provides the lower bound that prevents this metric from degenerating.

For distinct $x,y\in X$, define
\begin{align}
	\kappa(x,y)&=\max\{n\ge0:\text{some }(z,n)\in V_n
	\text{ satisfies }x,y\in B(z,2\alpha^n)\},
	\label{eq:kappa-xy}\\
	Q(x,y)&=\sup\{\pi((z,n)):(z,n)\in V_{\kappa(x,y)},
	\ x,y\in B(z,2\alpha^{\kappa(x,y)})\}.
	\label{eq:Qxy}
\end{align}
The set of levels in \eqref{eq:kappa-xy} is nonempty and finite when
$x\ne y$. Define
\begin{equation}\label{eq:Dchain}
	D_\sigma(x,y)=\inf_{x=x_0,\ldots,x_m=y}
	\sum_{i=1}^m Q(x_{i-1},x_i),
	\qquad D_\sigma(x,x)=0.
\end{equation}

The next proposition constructs the metric on $X$ and proves the diameter
estimate used in Section~\ref{sec:reduction-proof}.  It adapts
\cite[Lemma~2.3]{MTTree26} to the present filling, where $V_n$ need not be
finite.
Unlike the compact setting of \cite[Corollary~5.6]{MT26}, our application
uses the estimate only at $\mu$-almost every point, with the weights chosen
by random partitions.

\begin{proposition}\label{prop:filling}
	The function $D_\sigma$ is a metric, and
	$\id:(X,d)\to(X,D_\sigma)$ is quasisymmetric. Moreover,
	there is a constant $C>0$, depending only on $\alpha$ and $\eta$, such
	that, for every $x\in X$ and $n\ge1$,
	\begin{equation}\label{eq:filling-diameter}
		\diam_{D_\sigma}B_d(x,\tfrac14\alpha^n)
		\le C\prod_{j=1}^n\tau_j(x),
	\end{equation}
	where
	\begin{equation}\label{eq:tau}
		\tau_j(x)=
		\begin{cases}
			\eta,&\sigma(z,j)=0\text{ for every }z\in A_j
			\text{ with }d(x,z)\le30\alpha^j,\\
			1-\eta,&\text{otherwise.}
		\end{cases}
	\end{equation}
\end{proposition}

\begin{proof}
	\medskip
	\noindent{\itshape Claim 1. The function $D_\sigma$ is a metric. Moreover,
	there is a constant $C_1\ge1$, depending only on $\eta$, such that
	\begin{equation}\label{eq:DQ}
		C_1^{-1}Q(x,y)\le D_\sigma(x,y)\le Q(x,y)
		\qquad(x,y\in X,\ x\ne y).
	\end{equation}
	\par}
	\medskip

	We first prove the following elementary scale estimate:
	\begin{equation}\label{eq:k-distance}
		\alpha^{\kappa(x,y)+1}<d(x,y)<4\alpha^{\kappa(x,y)}.
	\end{equation}
	The upper bound follows from the definition of $\kappa(x,y)$. For the lower
	bound, put $\kappa=\kappa(x,y)$. If
	$d(x,y)\le\alpha^{\kappa+1}$, choose $z'\in A_{\kappa+1}$ with
	$d(x,z')<\alpha^{\kappa+1}$. Then
	$x,y\in B(z',2\alpha^{\kappa+1})$, contradicting the maximality of
	$\kappa$.
	
	The upper bound in \eqref{eq:DQ} follows by using the one-edge chain
	$(x,y)$ in \eqref{eq:Dchain}. For the lower bound, fix a finite chain
	$x=x_0,\ldots,x_m=y$ with distinct consecutive points, and put
	$\kappa_i=\kappa(x_{i-1},x_i)$. Choose
	\[
	L>\max\{\kappa(x,y),\kappa_1,\ldots,\kappa_m\}.
	\]
	For each $i$, choose $\xi_i\in A_L$ with
	$d(x_i,\xi_i)<\alpha^L$, and set $u_i=(\xi_i,L)$. Also choose a vertex
	$z_i\in V_{\kappa_i}$ occurring in \eqref{eq:Qxy}.
	Let $a_i^-=g(u_{i-1})_{\kappa_i}$ and
	$a_i^+=g(u_i)_{\kappa_i}$. Since
	$z_i$ and $u_{i-1}$ are both close to $x_{i-1}$,
	\[
	\begin{aligned}
		d(a_i^-,z_i)
		&<\sum_{j=\kappa_i}^{L-1}\alpha^j+\alpha^L+2\alpha^{\kappa_i}\\
		&<\left(\frac1{1-\alpha}+\alpha+2\right)\alpha^{\kappa_i}
		<4\alpha^{\kappa_i}.
	\end{aligned}
	\]
	The same estimate holds for $a_i^+$. Thus
	$a_i^-\sim z_i\sim a_i^+$.
	
	Construct a graph path $P_i$ from $u_{i-1}$ to $u_i$ by following the
	parent chains to $a_i^-$ and $a_i^+$ and joining these ancestors through
	$z_i$. Equations \eqref{eq:pi-horizontal-new} and
	\eqref{eq:pi-product} give
	\[
	\sum_{v\text{ on }P_i}\pi(v)
	\le C_2\pi(z_i),
	\qquad C_2=1+2\eta^{-2}.
	\]
	Indeed, \eqref{eq:pi-product} shows that the weight decreases by a factor
	of at most $1-\eta$ at each step along either vertical leg. Hence its total
	weight is at most
	$\pi(a_i^\pm)\sum_{j=0}^\infty(1-\eta)^j
	=\eta^{-1}\pi(a_i^\pm)\le\eta^{-2}\pi(z_i)$.
	Since $\pi(z_i)\le Q(x_{i-1},x_i)$, concatenating the paths
	$P_i$ and applying Lemma~\ref{lem:discrete-chain} yields
	\begin{equation}\label{eq:chain-Pi-lower}
		\sum_{i=1}^mQ(x_{i-1},x_i)
		\ge C_2^{-1}c_\eta\Pi(u_0,u_m).
	\end{equation}
	
	It remains to compare $\Pi(u_0,u_m)$ with $Q(x,y)$. Put
	$\kappa=\kappa(x,y)$ and choose a vertex $z\in V_\kappa$ occurring in
	\eqref{eq:Qxy} such that
	$\pi(z)>Q(x,y)/2$. Repeating the ancestor-center estimate used for
	$a_i^-$ with $\kappa_i$ replaced by $\kappa$ shows that
	$g(u_0)_\kappa\sim z\sim g(u_m)_\kappa$. Hence, if
	$h=k(u_0,u_m)$, then
	\begin{equation}\label{eq:h-lower}
		h\ge\max\{\kappa-1,0\},
	\end{equation}
	because, when $\kappa\ge1$, applying \eqref{eq:parent-short-chain} to
	$g(u_0)_\kappa\sim z\sim g(u_m)_\kappa$ gives
	$g(u_0)_{\kappa-1}\sim g(u_m)_{\kappa-1}$; when $\kappa=0$, the bound is
	immediate.
	
	Conversely, the level-$h$ ancestors of $u_0$ and $u_m$ are horizontal
	neighbors and hence less than $8\alpha^h$ apart, while $x$ and $y$ are each within
	$\alpha^h/(1-\alpha)$ of the corresponding ancestor; hence the triangle
	inequality gives
	\[
	\begin{aligned}
		d(x,y)
		&<\left(8+\frac2{1-\alpha}\right)\alpha^h
		<12\alpha^h.
	\end{aligned}
	\]
	Together with \eqref{eq:k-distance}, this forces
	\begin{equation}\label{eq:h-upper}
		h\le\kappa+1.
	\end{equation}
	Indeed, if $h\ge\kappa+2$, then
	$12\alpha^h\le12\alpha^{\kappa+2}<\alpha^{\kappa+1}$ because
	$\alpha<10^{-3}$, a contradiction. From
	\eqref{eq:h-lower}--\eqref{eq:h-upper} and \eqref{eq:pi-product},
	\[
	\pi(g(u_0)_h)\ge\eta\pi(g(u_0)_\kappa).
	\]
	Since $g(u_0)_\kappa\sim z$, it follows that
	\[
	\Pi(u_0,u_m)
	\ge\pi(g(u_0)_h)
	\ge\eta^2\pi(z)
	>\frac{\eta^2}{2}Q(x,y).
	\]
	Combining this with \eqref{eq:chain-Pi-lower} gives
	\[
	\sum_{i=1}^mQ(x_{i-1},x_i)
	\ge\frac{c_\eta\eta^2}{2C_2}Q(x,y).
	\]
	Taking the infimum over all chains proves \eqref{eq:DQ} with
	$C_1=2C_2/(c_\eta\eta^2)$.
	In particular, $D_\sigma(x,y)>0$ when $x\ne y$.

Symmetry follows from the symmetry of $Q$, and concatenating chains gives
the triangle inequality.
	Thus $D_\sigma$ is a metric, and Claim~1 follows.
	\nobreak\hspace{0.5em}$\square$\par\medskip
	\noindent{\itshape Claim 2. The identity
	$\id:(X,d)\to(X,D_\sigma)$ is quasisymmetric.\par}
	\medskip
	
	Let $x,a,b$ be distinct, and write
	$\kappa_a=\kappa(x,a)$ and $\kappa_b=\kappa(x,b)$.
	Choose a vertex $z\in V_{\kappa_a}$ occurring in the definition of $Q(x,a)$
	with $\pi(z)>Q(x,a)/2$, and choose any vertex
	$z'\in V_{\kappa_b}$ occurring
	in the definition of $Q(x,b)$. Then
	\begin{equation}\label{eq:Q-quotient}
		\frac{Q(x,a)}{Q(x,b)}<2\frac{\pi(z)}{\pi(z')}.
	\end{equation}
	If $\kappa_a\ge\kappa_b$, then
	\[
	\begin{aligned}
		d(g(z)_{\kappa_b},z')
		&\le d(g(z)_{\kappa_b},z)+d(z,x)+d(x,z')\\
		&<\sum_{j=\kappa_b}^{\kappa_a-1}\alpha^j
		+2\alpha^{\kappa_a}+2\alpha^{\kappa_b}\\
		&<\left(\frac1{1-\alpha}+4\right)\alpha^{\kappa_b}
		<8\alpha^{\kappa_b}.
	\end{aligned}
	\]
	Thus $g(z)_{\kappa_b}\sim z'$. If $\kappa_a<\kappa_b$, the same
	calculation gives $z\sim g(z')_{\kappa_a}$. Equations
	\eqref{eq:pi-horizontal-new}, \eqref{eq:Q-quotient}, and
	\eqref{eq:pi-product} now imply
	\begin{equation}\label{eq:Qratio}
		\frac{Q(x,a)}{Q(x,b)}
		\le2\eta^{-1}
		\begin{cases}
			(1-\eta)^{\kappa_a-\kappa_b},&\kappa_a\ge\kappa_b,\\
			\eta^{-(\kappa_b-\kappa_a)},&\kappa_a<\kappa_b.
		\end{cases}
	\end{equation}
	
	Suppose that $d(x,a)\le t\,d(x,b)$, and set
	\[
	\beta=\frac{\log(1/(1-\eta))}{\log(1/\alpha)},
	\qquad
	\gamma=\frac{\log(1/\eta)}{\log(1/\alpha)}.
	\]
	By \eqref{eq:k-distance},
	\[
	\alpha^{\kappa_a-\kappa_b}<\frac{4t}{\alpha}.
	\]
	Consequently,
	\[
	(1-\eta)^{\kappa_a-\kappa_b}
	\le\left(\frac4\alpha\right)^\beta t^\beta
	\quad\text{if }\kappa_a\ge\kappa_b,
	\]
	whereas
	\[
	\eta^{-(\kappa_b-\kappa_a)}
	\le\left(\frac4\alpha\right)^\gamma t^\gamma
	\quad\text{if }\kappa_a<\kappa_b.
	\]
	Using Claim~1 and \eqref{eq:Qratio}, we obtain
	\[
	\frac{D_\sigma(x,a)}{D_\sigma(x,b)}
	\le C_3(t^\beta+t^\gamma),
	\]
	where
	\[
	C_3=2C_1\eta^{-1}
	\max\left\{
	\left(\frac4\alpha\right)^\beta,
	\left(\frac4\alpha\right)^\gamma
	\right\}.
	\]
	Since $t\mapsto C_3(t^\beta+t^\gamma)$ is a homeomorphism of
	$[0,\infty)$ onto itself, Claim~2 follows.
	\nobreak\hspace{0.5em}$\square$\par\medskip
	\noindent{\itshape Claim 3. For every $x\in X$ and $n\ge1$,
	\eqref{eq:filling-diameter} holds.\par}
	\medskip
	
	Choose $u=(z,n)\in V_n$ with $d(x,z)<\alpha^n$. If
	$y,y'\in B_d(x,\alpha^n/4)$ are distinct, then
	\[
	d(y,z)<\tfrac54\alpha^n<2\alpha^n,
	\]
	and the same holds for $y'$. Hence $\kappa(y,y')\ge n$. Put
	$\kappa=\kappa(y,y')$, and let $\widetilde u=(w,\kappa)$ be any vertex occurring in
	the definition of $Q(y,y')$. Then
	\[
	\begin{aligned}
		d(g(\widetilde u)_n,u)
		&\le d(g(\widetilde u)_n,w)+d(w,y)+d(y,x)+d(x,z)\\
		&<\sum_{\ell=n}^{\kappa-1}\alpha^\ell
		+2\alpha^\kappa+\tfrac14\alpha^n+\alpha^n\\
		&<\left(\frac1{1-\alpha}+\frac{13}{4}\right)\alpha^n
		<8\alpha^n.
	\end{aligned}
	\]
	Thus $g(\widetilde u)_n\sim u$. By
	\eqref{eq:pi-horizontal-new} and \eqref{eq:pi-product},
	\[
	\pi(\widetilde u)
	\le(1-\eta)^{\kappa-n}\pi(g(\widetilde u)_n)
	\le\eta^{-1}\pi(u).
	\]
	Taking the supremum over all such $\widetilde u$ and using the upper
	bound in \eqref{eq:DQ}, we obtain
	\begin{equation}\label{eq:ball-pi}
		\diam_{D_\sigma}B_d(x,\tfrac14\alpha^n)
		\le\eta^{-1}\pi(u).
	\end{equation}
	
	For $1\le j\le n$, the parent displacement estimates give
	\[
	d(x,g(u)_j)
	<\sum_{\ell=j}^n\alpha^\ell
	<\frac{\alpha^j}{1-\alpha}
	<2\alpha^j.
	\]
	Suppose the first case of \eqref{eq:tau} holds at level $j$. If
	$g(u)_j\sim v\sim v_1\sim v_2$, then
	\[
	d(x,v_2)<2\alpha^j+3\cdot8\alpha^j
	=26\alpha^j<30\alpha^j.
	\]
	Therefore $\sigma(v_2)=0$. It follows from \eqref{eq:mweight} that
	$a(v)=\eta$ for every $v\sim g(u)_j$, and hence
	\[
	\rho(g(u)_j)=\eta=\tau_j(x)
	\]
	by \eqref{eq:rho-bounds}. In the second case of \eqref{eq:tau}, the
	same bounds give
	$\rho(g(u)_j)\le1-\eta=\tau_j(x)$. Thus
	\[
	\pi(u)=\prod_{j=1}^n\rho(g(u)_j)
	\le\prod_{j=1}^n\tau_j(x).
	\]
	Combining this with \eqref{eq:ball-pi} proves Claim~3 with
	$C=\eta^{-1}$.
	\nobreak\hspace{0.5em}$\square$\par\medskip
	The proposition follows from Claims~1--3.
\end{proof}

\section{Proof of the dimension-reduction theorem}\label{sec:reduction-proof}

In this section we prove Theorem~\ref{thm:reduction}.
Subsection~\ref{subsec:measure-tools} relates Hausdorff dimension to local
dimension and constructs random partitions with an explicit bound on the
probability that a ball is not contained in the partition cell containing its
center.  Subsection~\ref{subsec:bounded-complete} uses these partitions to
choose the filling weights when the metric space $(X,d)$ is bounded and complete.
Subsection~\ref{subsec:general-case} then removes the assumptions of
boundedness and completeness.

\subsection{Local dimension and random partitions}
\label{subsec:measure-tools}

This subsection proves two measure-theoretic lemmas used later, one concerning
local dimension and the other concerning random partitions.

For a finite measure $\mu$ on $(X,d)$, the \emph{lower and upper local
dimensions} of $\mu$ at $x\in X$ are defined by
\[
\underline{\dim}_{\mathrm{loc}}\mu(x)
=\liminf_{r\downarrow0}
\frac{\log\mu(B_d(x,r))}{\log r},
\qquad
\overline{\dim}_{\mathrm{loc}}\mu(x)
=\limsup_{r\downarrow0}
\frac{\log\mu(B_d(x,r))}{\log r}.
\]
If these two quantities agree, their common value is called the \emph{local
dimension} of $\mu$ at $x$.  These definitions are needed only on
$\spt\mu$, which has full measure; their values off the support are
irrelevant.  We use the following relation
between the lower local dimension and the Hausdorff dimension of a measure.
\begin{lemma}\label{lem:localdim}
If $\mu$ is a finite measure on a separable metric space $(X,d)$,
then
\begin{equation}\label{eq:localdim}
	\dimH^d\mu
	 =\esssup_{x\in X}\underline{\dim}_{\mathrm{loc}}\mu(x),
\end{equation}
where the essential supremum is taken with respect to $\mu$.
\end{lemma}

\begin{proof}
	For measures on Euclidean spaces, \eqref{eq:localdim} is standard; see
	\cite[Proposition~10.3]{F97}.  The same proof works in a separable metric
	space.  For $s$ above the right-hand side of \eqref{eq:localdim}, applying
	the metric $5r$-covering lemma \cite[Theorem~1.2]{H01} to balls satisfying
	$\mu(B(x,r))\ge r^s$ gives the upper bound.  For $s$ below the right-hand
	side, the definition of lower local dimension gives a set of positive
	measure on which $\mu(B(x,r))\le r^s$ uniformly for all sufficiently small
	$r$; the definition of Hausdorff measure then gives the lower bound.
\end{proof}

The following lemma constructs a random Borel partition of a separable metric
space equipped with a probability measure that has full support. A \emph{Borel partition} is a pairwise disjoint collection of
Borel sets whose union is the whole space; a \emph{random Borel partition}
is a partition-valued random variable.  We write $\mathcal P(x)$ for the cell
containing $x$ and $\mathbb P$ for probability with respect to this random
choice.

\begin{lemma}\label{lem:partition}
Let $(X,d)$ be a separable metric space and let $\mu$ be a Borel probability
measure with full support. For every $\Delta>0$ there is a random Borel partition
$\mathcal P$ of $X$ whose cells have diameter at most $\Delta$ and such
that, for $0<t\le\Delta/8$ and $x\in X$,
\begin{equation}\label{eq:partition-bound}
 \mathbb P[B_d(x,t)\not\subset\mathcal P(x)]
 \le \frac{8t}{\Delta}
       \log\frac{\mu(B_d(x,\Delta))}
                    {\mu(B_d(x,\Delta/8))}.
\end{equation}
\end{lemma}

\begin{proof}
Choose $R$ uniformly from $[\Delta/4,\Delta/2]$.  Independently of
$R$, let
$Z_1,Z_2,\ldots$ be independent samples with law $\mu$. 

We first verify that the sequence $\{Z_i:i\ge1\}$ is almost surely
dense in $X$. Let $\{U_j:j\ge1\}$ be a countable base consisting of
nonempty open sets. Since $\mu$ has full support,
$
\mu(U_j)>0$ for every $j\ge1$. By the independence of the samples,
\[
\begin{aligned}
	\mathbb P\bigl[Z_i\notin U_j\text{ for every }i\ge1\bigr]
	&=\lim_{N\to\infty}
	\mathbb P\bigl[Z_1,\ldots,Z_N\notin U_j\bigr]\\
	&=\lim_{N\to\infty}(1-\mu(U_j))^N=0.
\end{aligned}
\]
It follows that, with probability one, every $U_j$ contains at least one
sample $Z_i$.  Thus $\{Z_i:i\ge1\}$ is dense in $X$ almost surely.

Fix such a dense sequence $\{Z_i:i\ge1\}$. For $x\in X$,
define
$
I_R(x)=\min\{i\ge1:d(x,Z_i)<R\}.
$ 
For each $i\ge1$, define
$
P_i=\{x\in X:I_R(x)=i\}.
$
Equivalently,
\[
P_i
=B(Z_i,R)\setminus\bigcup_{\ell=1}^{i-1}B(Z_\ell,R).
\]
Thus every $P_i$ is Borel. Moreover, the sets $P_i$ are pairwise
disjoint and their union is $X$.
Define
\[
\mathcal P=\{P_i:i\ge1,\ P_i\ne\varnothing\}.
\]
Consequently, $\mathcal P$ is a Borel partition of $X$ whose cells
have diameter at most $\Delta$.

For $x\in X$, write
$
\mathcal P(x)=P_{I_R(x)}
$
for the unique cell containing $x$.  Fix $x$ and $t\le\Delta/8$, put
$m(r)=\mu(B(x,r))$, and let $j=I_{R+t}(x)$.  Suppose that
$Z_j\in B(x,R-t)$.  For every $y\in B(x,t)$,
\[
d(y,Z_j)\le d(y,x)+d(x,Z_j)<t+(R-t)=R.
\]
Thus $y\in B(Z_j,R)$. On the other hand, if $i<j$, then
$Z_i\notin B(x,R+t)$. Therefore
\[
d(y,Z_i)\ge d(x,Z_i)-d(x,y)>R+t-t=R,
\]
so $y\notin B(Z_i,R)$. Hence $I_R(y)=j$ for every $y\in B(x,t)$.
In particular, all points of $B(x,t)$ belong to the same partition
cell as $x$, and 
$
B(x,t)\subset\mathcal P(x).
$ 
Note that 
\[
\begin{aligned}
\mathbb{P}[Z_j\in B(x,R-t)\mid R]
&=\sum_{i=1}^\infty
\mathbb{P}\bigl[Z_1,\ldots,Z_{i-1}\notin B(x,R+t),
Z_i\in B(x,R-t)\mid R\bigr]\\
&=\sum_{i=1}^\infty (1-m(R+t))^{i-1}m(R-t)
=\frac{m(R-t)}{m(R+t)}.
\end{aligned}
\]
Therefore
\[
 \mathbb P[B(x,t)\not\subset\mathcal P(x)\mid R]
 \le1-\frac{m(R-t)}{m(R+t)}
 \le\log\frac{m(R+t)}{m(R-t)}.
\]  
If
$F(r)=\log m(r)$, monotonicity gives
\begin{align*}
 \int_{\Delta/4}^{\Delta/2}[F(r+t)-F(r-t)]\,dr
 &\le 2t[F(\Delta/2+t)-F(\Delta/4-t)]\\
 &\le 2t\log\frac{m(\Delta)}{m(\Delta/8)}.
\end{align*}
Multiplication by the density $4/\Delta$ of $R$ proves
\eqref{eq:partition-bound}.
\end{proof}

\subsection{The bounded complete case}
\label{subsec:bounded-complete}

In this subsection we prove Theorem~\ref{thm:reduction} when $(X,d)$ is
bounded and complete.  The main point is to arrange that most scales receive
the smaller factor $\eta$, while the exceptional scales have sufficiently
small upper density.

\begin{proof}[Proof of Theorem~\ref{thm:reduction} in the bounded complete case]
Assume first that $X$ is bounded and complete, and rescale so that
$\diam X<1$.  Set $s_\mu=\dimH^d\mu<\infty$. Fix $p>0$, $\varepsilon>0$, and a
constant $K>40$.  Choose a large integer $M$ such that,
with $\alpha=M^{-1}<\alpha_0$, where $\alpha_0$ is the constant fixed in
Subsection~\ref{subsec:nets},
\begin{equation}\label{eq:delta}
 M\ge32K,
 \qquad
 \delta_M:=\frac{32K(s_\mu+\varepsilon)\log M}{M}<\frac12.
\end{equation}

At every level $n$, independently apply Lemma~\ref{lem:partition} with
\begin{equation}\label{eq:Delta-n}
 \Delta_n=\frac{\alpha^{n-1}}4
\end{equation}
and denote the partition by $\mathcal P_n$.  For each vertex $(z,n)\in V_n$, define
\begin{equation}\label{eq:sigma}
 \sigma(z,n)=\mathbf 1_{\{B(z,9\alpha^n) \not\subset\mathcal P_n(z)\}}.
\end{equation}
This assignment is admissible. Indeed, suppose that
$u_i=(x_i,n)\in V_n$ satisfies \eqref{eq:annular-endpoints} and that
$\sigma(u_i)=0$ for every $0\le i\le N$. Since
$
d(x_{i-1},x_i)<8\alpha^n<9\alpha^n,
$
we have
\[
x_i\in B(x_{i-1},9\alpha^n)
\subset\mathcal P_n(x_{i-1}),
\]
and hence
$
\mathcal P_n(x_i)=\mathcal P_n(x_{i-1}).
$
Thus all $x_i$ belong to one cell, so
\[
d(x_0,x_N)\le\Delta_n=\frac{\alpha^{n-1}}4.
\]
On the other hand, \eqref{eq:annular-endpoints} gives
\[
d(x_0,x_N)>
\alpha^{n-1}-8\alpha^n
=\left(1-\frac8M\right)\alpha^{n-1}
>\frac{\alpha^{n-1}}4
\]
The last inequality follows from \eqref{eq:delta}.
This is a contradiction. 

As in the proof of Lemma~\ref{lem:partition}, let $R_n$ and
$Z_{n,1},Z_{n,2},\ldots$ be the random radius and samples used to construct
$\mathcal P_n$.  For $x\in X$, let $I_n(x)$ be the first $i$ such
that $Z_{n,i}\in B(x,R_n+K\alpha^n)$ and set
\[
 J_n(x)=\mathbf 1_{\{d(x,Z_{n,I_n(x)}) \ge R_n-K\alpha^n\}}.
\]
The events $\{I_n(x)=i\}$ are Borel in the sample variables and $x$, so
$J_n(x)$ is jointly measurable.  It is related to the factor $\tau_n(x)$ in
\eqref{eq:tau} through the implication
\begin{equation}\label{eq:jn}
J_n(x)=0\quad\Longrightarrow\quad\tau_n(x)=\eta.
\end{equation}
Indeed, if $J_n(x)=0$, then the proof of
Lemma~\ref{lem:partition} gives
$
B(x,K\alpha^n)\subset\mathcal P_n(x).
$
If $(z,n)\in V_n$ satisfies $d(x,z)\le 30\alpha^n$, then, since
$K>40$,
\[
B(z,9\alpha^n)
\subset B(x,39\alpha^n)
\subset B(x,K\alpha^n)
\subset\mathcal P_n(x).
\]
In particular, $z\in\mathcal P_n(x)$, so
$\mathcal P_n(z)=\mathcal P_n(x)$ and
$
B(z,9\alpha^n)\subset\mathcal P_n(z).
$
It follows from \eqref{eq:sigma} that
\[
\sigma(z,n)=0
\qquad\text{whenever }d(x,z)\le 30\alpha^n.
\]
Hence \eqref{eq:jn} holds, and therefore
\begin{equation}\label{eq:tau-count}
 \#\{1\le j\le n:\tau_j(x)=\eta\}
 \ge n-\sum_{j=1}^nJ_j(x).
\end{equation}

Writing $\E$ for expectation over all the random partitions, note that
$M\ge32K$ gives $K\alpha^n\le\Delta_n/8$.  The calculation in the proof of
Lemma~\ref{lem:partition}, with $t=K\alpha^n$, gives
\begin{equation}\label{eq:pn}
 p_n(x):=\E J_n(x)
 \le\frac{32K}{M}
    \log\frac{\mu(B(x,\Delta_n))}
                  {\mu(B(x,\Delta_n/8))}.
\end{equation}
By \eqref{eq:Delta-n},
$\Delta_{n+1}=\Delta_n/M\le\Delta_n/8$.  Summing \eqref{eq:pn} and
telescoping therefore give
\begin{equation}\label{eq:telescope}
 \sum_{j=1}^n p_j(x)
 \le\frac{32K}{M}
    \log\frac{\mu(B(x,\Delta_1))}
                  {\mu(B(x,\Delta_{n+1}))}.
\end{equation}

For fixed $x$, the variables $J_n(x)$ are independent and bounded.
Writing $\operatorname{Var}$ for variance, we have
$\sum_n\operatorname{Var}(J_n(x))/n^2<\infty$, so Kolmogorov's
strong law gives
\[
 \frac1n\sum_{j=1}^n[J_j(x)-p_j(x)]\longrightarrow0
\]
with probability one. By joint measurability and Fubini's theorem, there is a
deterministic choice of all partitions for which
\begin{equation}\label{eq:slln}
 \frac1n\sum_{j=1}^n[J_j(x)-p_j(x)]\longrightarrow0
 \quad\text{for $\mu$-almost every }x.
\end{equation}

Fix such an $x$ with $\underline{\dim}_{\mathrm{loc}}\mu(x)\le s_\mu$, as holds almost
everywhere by Lemma~\ref{lem:localdim}.
There are arbitrarily small radii $r$ for which
$\mu(B(x,r))\ge r^{s_\mu+\varepsilon}$.  For each such $r$, choose $n$ so that
$\tfrac14\alpha^{n+1}<r\le\tfrac14\alpha^n$.  Hence there are arbitrarily large $n$ such
that
\begin{equation}\label{eq:masssubsequence}
 \mu(B(x,\tfrac14\alpha^n))
 \ge \alpha^{s_\mu+\varepsilon}(\tfrac14\alpha^n)^{s_\mu+\varepsilon}.
\end{equation}
Since $\tfrac14\alpha^n=\Delta_{n+1}$, equations
\eqref{eq:telescope} and \eqref{eq:slln} give, along this subsequence,
with $o(1)$ denoting a quantity that tends to zero,
\begin{equation}\label{eq:badfraction}
 \frac1n\sum_{j=1}^nJ_j(x)\le\delta_M+o(1).
\end{equation}

Apply Proposition~\ref{prop:filling}.  By \eqref{eq:tau-count} and
\eqref{eq:badfraction}, at least $(1-\delta_M-o(1))n$ scales contribute the
factor $\eta$, and every remaining scale contributes at most $1-\eta<1$.
Proposition~\ref{prop:filling} gives
\begin{equation}\label{eq:ball-inclusion}
 B_d(x,\tfrac14\alpha^n)
 \subset B_{D_\sigma}\!\left(x,
 2C\eta^{(1-\delta_M-o(1))n}\right)
\end{equation}
along the same subsequence.  Write
$\underline{\dim}_{\mathrm{loc}}^{D_\sigma}\mu$ for the lower local dimension computed using
$D_\sigma$-balls.  Combining
\eqref{eq:masssubsequence} and \eqref{eq:ball-inclusion} yields
\begin{equation}\label{eq:newlocaldim}
 \underline{\dim}_{\mathrm{loc}}^{D_\sigma}\mu(x)
 \le\frac{(s_\mu+\varepsilon)\log M}
          {(1-\delta_M)\log(1/\eta)}
 \quad\text{for $\mu$-almost every }x.
\end{equation}
Lemma~\ref{lem:localdim}, applied to \eqref{eq:newlocaldim}, gives the same
upper bound for
$\dimH^{D_\sigma}\mu$.  After $M$ has been fixed, $\eta$ may be chosen
arbitrarily small, so the right side is less than the prescribed $p$.
\end{proof}

\subsection{The general case}
\label{subsec:general-case}

The preceding subsection proves the theorem for bounded complete
spaces. We now pass first to complete unbounded spaces by
\emph{sphericalization}, which adds one point at infinity, and then to
incomplete spaces by taking a completion.

For positive quantities $A$ and $B$, we write
$A\asymp B$ if there is a constant $C\ge1$, independent of the relevant
variables, such that $C^{-1}B\le A\le CB$.
Two metrics $d_1$ and $d_2$ are \emph{bi-Lipschitz equivalent} on a set $E$
if $d_1(x,y)\asymp d_2(x,y)$ for all $x,y\in E$.

\begin{lemma}\label{lem:sphericalization}
Let $(X,d)$ be complete, unbounded, and separable, and fix $o\in X$.
There is a bounded complete metric $\widehat d$ on
$\widehat X=X\cup\{\infty\}$ satisfying
\begin{align}
 \widehat d(x,y)&\asymp
 \frac{d(x,y)}{(1+d(x,o))(1+d(y,o))},\label{eq:spherical-xy}\\
 \widehat d(x,\infty)&\asymp\frac1{1+d(x,o)}.
 \label{eq:spherical-inf}
\end{align}
Suppose $\widehat D$ is another metric on $\widehat X$ for which
$\id:(\widehat X,\widehat d)\to(\widehat X,\widehat D)$ is
quasisymmetric.  Then there exists a metric $D$ on $X$ such that
\begin{equation}\label{eq:flatten}
	D(x,y)\asymp
	\frac{\widehat D(x,y)}
	{\widehat D(x,\infty)\widehat D(y,\infty)},
	\qquad x,y\in X,
\end{equation}
and
$
\id:(X,d)\longrightarrow(X,D)
$
is quasisymmetric. Moreover, $D$ and $\widehat D$ are bi-Lipschitz
equivalent on every set $E\subset X$ on which
$\widehat D(\,\cdot\,,\infty)$ is bounded away from zero.
\end{lemma}

\begin{proof}
	For $x\in X$, let
	$
	s(x)=1+d(x,o)$. Write $s(\infty)=\infty$ and $1/s(\infty)=0$. 
	Define a symmetric kernel $\widehat q$ on
	$\widehat X=X\cup\{\infty\}$ by
	\[
	\widehat q(x,y)=\frac{d(x,y)}{s(x)s(y)},
	\qquad
	\widehat q(x,\infty)=\widehat q(\infty,x)=\frac1{s(x)},
	\qquad
	\widehat q(\infty,\infty)=0
	\]
	for $x,y\in X$. For $x,y\in\widehat X$, define
	\begin{equation}\label{eq:spherical-chain}
		\widehat d(x,y)
		=
		\inf\left\{
		\sum_{i=1}^N \widehat q(x_{i-1},x_i):
		N\ge1,\;
		x_0=x,\;
		x_N=y,\;
		x_i\in\widehat X
		\right\}.
	\end{equation}
	Thus \eqref{eq:spherical-chain} is the infimum of the $\widehat q$-lengths of all
	finite chains in $\widehat X$ joining $x$ to $y$.  To prove
	\eqref{eq:spherical-xy}--\eqref{eq:spherical-inf}, it suffices to verify
	that, for $x,y\in\widehat X$,
	\[
	\widehat d(x,y)\ge \frac14\widehat q(x,y).
	\]
	If $x=y$, this inequality is trivial.  If $x\ne y$, symmetry allows us to
	assume $s(x)\le s(y)$, and then $x\in X$.  Consider a chain from $x$ to
	$y$.  Since
\begin{equation}\label{eq:reciprocal-lipschitz}
 \left|\frac1{s(u)}-\frac1{s(v)}\right|
 \le \widehat q(u,v),
\end{equation}
if the chain first reaches a point $z$ with $s(z)\ge2s(x)$, then \eqref{eq:reciprocal-lipschitz} gives
\[
\sum_i \widehat q(x_{i-1},x_i)
\ge \frac1{s(x)}-\frac1{s(z)}
\ge\frac1{2s(x)}
\ge\frac14\widehat q(x,y).
\]
Here we used $\widehat q(x,y)\le2/s(x)$.  If the
chain never reaches such a point, none of its vertices is $\infty$, all
denominators are at most $4s(x)^2$, and the triangle inequality gives
\[
 \sum_i\widehat q(x_{i-1},x_i)
 \ge\frac{d(x,y)}{4s(x)^2}\ge\frac14\widehat q(x,y).
\]
Taking the infimum over
all chains yields
$
\widehat d(x,y)\ge\frac14\widehat q(x,y).
$
The definition gives symmetry and the triangle inequality, while the preceding
lower bound separates distinct points; hence $\widehat d$ is a metric.  It is
bounded because $\widehat q\le2$.  To check completeness, let $(x_j)$
be a $\widehat d$-Cauchy sequence.  If $s(x_j)$ is unbounded along a
subsequence, \eqref{eq:spherical-inf} and the Cauchy property imply
$x_j\to\infty$.  Otherwise $s(x_j)$ is eventually bounded, and
\eqref{eq:spherical-xy} shows that $(x_j)$ is $d$-Cauchy.  It then converges
to some $x\in X$ because $(X,d)$ is complete.  Since
$\widehat d(x_j,x)\le \widehat q(x_j,x)\le d(x_j,x)$, it also converges to $x$ in
$\widehat d$.  Thus
$(\widehat X,\widehat d)$ is a bounded complete metric space.

For the reverse construction, let
\[
r(x)=\widehat D(x,\infty),\qquad
i(x,y)=\frac{\widehat D(x,y)}{r(x)r(y)},
\qquad x,y\in X,
\]
and define $D$ to be the chain metric induced by $i$.
The triangle inequality for $\widehat D$ gives
\[
\left|\frac1{r(u)}-\frac1{r(v)}\right|
\le i(u,v).
\]
Repeating the preceding chain argument, with $r$ in place of $s$, gives
\[
\frac14 i(x,y)\le D(x,y)\le i(x,y),
\]
and hence proves \eqref{eq:flatten}.  By its definition, $D$ is symmetric and
satisfies the triangle inequality, while the lower bound separates distinct
points; hence $D$ is a metric.

We next prove that
$
\id:(X,d)\longrightarrow(X,D)
$
is quasisymmetric.  By \cite[Proposition~3.1(ii)]{B11}, every quasisymmetric
map satisfies the following four-point inequality for some homeomorphism
$\Theta:[0,\infty)\to[0,\infty)$ such that
\[
\frac{\widehat D(x,a)\widehat D(b,\infty)}
{\widehat D(x,b)\widehat D(a,\infty)}
\le
\Theta\left(
\frac{\widehat d(x,a)\widehat d(b,\infty)}
{\widehat d(x,b)\widehat d(a,\infty)}
\right)
\]
for all distinct $x,a,b\in X$.

By the comparison $i/4\le D\le i$,
\[
\begin{aligned}
	\frac{D(x,a)}{D(x,b)}
	\le4\frac{i(x,a)}{i(x,b)}
	=4\frac{\widehat D(x,a)\widehat D(b,\infty)}
	{\widehat D(x,b)\widehat D(a,\infty)}.
\end{aligned}
\]
Moreover, the estimates $\widehat q/4\le\widehat d\le \widehat q$ give
\[
\begin{aligned}
	\frac{\widehat d(x,a)\widehat d(b,\infty)}
	{\widehat d(x,b)\widehat d(a,\infty)}
	\le16
	\frac{\widehat q(x,a)\widehat q(b,\infty)}
	{\widehat q(x,b)\widehat q(a,\infty)}
	=16\frac{d(x,a)}{d(x,b)}.
\end{aligned}
\]
Consequently,
\[
\frac{D(x,a)}{D(x,b)}
\le
4\Theta\left(16\frac{d(x,a)}{d(x,b)}\right).
\]
Since $t\mapsto4\Theta(16t)$ is a homeomorphism of $[0,\infty)$
onto itself, the identity $\id:(X,d)\to(X,D)$ is quasisymmetric.

Since $\widehat d$ is bounded and the identity from $\widehat d$ to
$\widehat D$ is quasisymmetric, $\widehat D$ is bounded.  Hence there is
$C<\infty$ such that $\widehat D(x,\infty)\le C$ for every $x\in X$.
Finally, suppose that
\[
0<c\le\widehat D(x,\infty),
\qquad x\in E.
\]
For $x,y\in E$, the definition of $i$ gives
\[
C^{-2}\widehat D(x,y)
\le i(x,y)
\le c^{-2}\widehat D(x,y).
\]
Combining this with $i/4\le D\le i$, we obtain
\[
\frac1{4C^2}\widehat D(x,y)
\le D(x,y)
\le\frac1{c^2}\widehat D(x,y),
\qquad x,y\in E.
\]
Hence $D$ and $\widehat D$ are bi-Lipschitz equivalent on $E$.

\end{proof}

\begin{proof}[Proof of Theorem~\ref{thm:reduction} in the general case]
Assume now that $(X,d)$ is complete and unbounded.  Let
	$
	\widehat X=X\cup\{\infty\}
	$
	and let $\widehat d$ be the metric given by
	Lemma~\ref{lem:sphericalization}. Extend $\mu$ to $\widehat X$ by setting
	$
	\widehat\mu(A)=\mu(A\cap X).
	$
	The measure $\widehat\mu$ has full support. Indeed, $\widehat d$ and
	$d$ induce the same topology on $X$, while every
	$\widehat d$-neighborhood of $\infty$ contains the complement of a
	bounded $d$-ball. Since $X$ is unbounded and $\mu$ has full support,
	such a neighborhood has positive measure.
	
	On every bounded $d$-ball, the metrics $d$ and $\widehat d$ are
	bi-Lipschitz equivalent by \eqref{eq:spherical-xy}.  Since $X$ is the
	countable union of bounded balls, the countable-union formula shows that
	$\dimH^{\widehat d}E=\dimH^dE$ for every $E\subset X$.  Moreover,
	$\widehat\mu(\{\infty\})=0$.  Therefore
	\[
	\dimH^{\widehat d}\widehat\mu
	= \dimH^d\mu<\infty.
	\]
	The bounded complete case therefore provides a metric $\widehat D$ on
	$\widehat X$ such that
	$
	\id:(\widehat X,\widehat d)
	\longrightarrow(\widehat X,\widehat D)
	$
	is quasisymmetric and
	$
	\dimH^{\widehat D}\widehat\mu\le p.
	$
	
	Let $D$ be the metric associated with $\widehat D$ by
	Lemma~\ref{lem:sphericalization}. Then
	$
	\id:(X,d)\longrightarrow(X,D)
	$
	is quasisymmetric. For $m\ge1$, set
	\[
	E_m=
	\left\{x\in X:
	m^{-1}\le\widehat D(x,\infty)\le m
	\right\}.
	\]
	The sets $E_m$ cover $X$, and $D$ and $\widehat D$ are bi-Lipschitz
	equivalent on each $E_m$.
	
	Fix $\varepsilon>0$ and choose a
	full $\widehat\mu$-measure set $A\subset\widehat X$ such that
	$
	\dimH^{\widehat D}A\le p+\varepsilon.
	$
	Then
	\[
	A\cap X=\bigcup_{m=1}^{\infty}(A\cap E_m),
	\]
	and bi-Lipschitz invariance, followed by the same countable-union formula,
	gives
	\[
	\dimH^D(A\cap X)
	=\sup_{m\ge1}\dimH^D(A\cap E_m)
	\le p+\varepsilon.
	\]
	Since $A\cap X$ has full $\mu$-measure, it follows that
	$\dimH^D\mu\le p+\varepsilon$. Letting $\varepsilon \to 0$ yields
	\[
	\dimH^D\mu\le p.
	\]
	
	Finally, suppose that $(X,d)$ is not complete. Let
	$(\overline X,\overline d)$ be its completion, and let
	$
	\iota:X\to \overline X
	$
	be the isometric inclusion. Define
	$
	\overline\mu=\iota_\#\mu.
	$
	Then $\overline X$ is separable and $\overline\mu$ has full support. We also have
	$
	\dimH^{\overline d}\overline\mu
	\le\dimH^d\mu<\infty.
	$
	To see this, fix $s>\dimH^d\mu$ and choose a Borel set
	$A\subset X$ such that $\mu(A)=1$ and $\dimH^d A<s$.  Choose
	$t$ with $\dimH^d A<t<s$.  Then
	$\mathcal H^t_{\overline d}(\iota(A))=0$.  By Borel regularity, there is a
	Borel set $\overline A\subset\overline X$ containing $\iota(A)$ such that
	$\mathcal H^t_{\overline d}(\overline A)=0$.  Thus
	$\overline\mu(\overline A)=1$, $\dimH^{\overline d}\overline A<s$, and
	$\dimH^{\overline d}\overline\mu\le s$.
	Since $s>\dimH^d\mu$ is arbitrary,
	$
	\dimH^{\overline d}\overline\mu\le\dimH^d\mu.
	$
	
	Apply the complete space case to
	$(\overline X,\overline d,\overline\mu)$. We obtain a metric
	$\overline D$ on $\overline X$ such that
	$
	\id:(\overline X,\overline d)
	\longrightarrow(\overline X,\overline D)
	$
	is quasisymmetric and
	$
	\dimH^{\overline D}\overline\mu\le p.
	$
	Let $D$ be the restriction of $\overline D$ to $X$. The restriction
	of a quasisymmetric map is quasisymmetric, so
	\[
	\id:(X,d)\longrightarrow(X,D)
	\]
	is quasisymmetric. Moreover, intersecting any full
	$\overline\mu$-measure Borel set in $\overline X$ with $X$ gives a
	full $\mu$-measure Borel set in $X$. Therefore
	\[
	\dimH^D\mu
	\le\dimH^{\overline D}\overline\mu
	\le p.
	\]
	This completes the proof.
\end{proof}

\section{An alternative proof of Corollary~\ref{cor:doubling}}\label{sec:doubling}
This section gives an independent proof of Corollary~\ref{cor:doubling}.
It combines Assouad's embedding with Romney's dimension
reduction for Lebesgue measure.

\begin{lemma}\label{lem:translation}
Let \(n\ge1\), let \((S,d)\) be a metric space, let \(\mu\) be a
\(\sigma\)-finite Borel measure on \(S\), let \(\varphi:S\to\R^n\) be a Borel
map, and let \(U\subset\R^n\) be a Borel set with \(\mathcal L^n(U)=0\). Then
for $\mathcal L^n$-almost every $t\in[0,1]^n$,
\[
  \mu\{x\in S:\ \varphi(x)+t\in U\}=0.
\]
\end{lemma}

\begin{proof}
Choose Borel sets \(S_j\subset S\) such that
\(S=\bigcup_{j=1}^{\infty}S_j\) and \(\mu(S_j)<\infty\) for every \(j\).

Fix \(j\in \N\). By Tonelli's theorem,
\[
\begin{aligned}
  \int_{[0,1]^n}\mu\{x\in S_j:\ \varphi(x)+t\in U\}\,dt
  &=
  \int_{[0,1]^n}\int_{S_j}\mathbf 1_U(\varphi(x)+t)\,d\mu(x)\,dt  \\
  &=
  \int_{S_j}\int_{[0,1]^n}\mathbf 1_U(\varphi(x)+t)\,dt\,d\mu(x)  \\
  &=
  \int_{S_j}\mathcal L^n\bigl([0,1]^n\cap(U-\varphi(x))\bigr)\,d\mu(x)
  =
  0.
\end{aligned}
\]
Hence there is a set \(N_j\subset[0,1]^n\) with
\(\mathcal L^n(N_j)=0\) such that, for every
\(t\in[0,1]^n\setminus N_j\),
\begin{equation}\label{eq:translation}
\mu\{x\in S_j:\ \varphi(x)+t\in U\}=0.
\end{equation}
Set \(N=\bigcup_{j=1}^{\infty}N_j\). Then \(\mathcal L^n(N)=0\), and for
each \(t\in[0,1]^n\setminus N\), identity \eqref{eq:translation} holds for every
\(j\in\N\). Since \(S=\bigcup_{j=1}^{\infty}S_j\), this gives
\[
  \mu\{x\in S:\ \varphi(x)+t\in U\}=0.
\]
This proves the lemma.
\end{proof}

Given metric spaces $(X,d_X)$ and $(Y,d_Y)$, an injective map $f:X\to Y$
is a \emph{quasisymmetric embedding} if the induced map
\[
f: (X,d_X) \longrightarrow (f(X),\, d_Y|_{f(X)})
\]
is quasisymmetric.

\begin{lemma} \label{lem:euclidean-reduction}
Let \(n\ge1\), let \((S,d)\) be a metric space, and let \(\mu\) be a
\(\sigma\)-finite Borel measure on \(S\) with \(\spt\mu=S\).
Suppose that there is a quasisymmetric embedding
\(\varphi:S\to\R^n\). Then, for every \(p>0\), there exist a metric space
\(Y\) and a quasisymmetric homeomorphism \(h:S\to h(S)\subset Y\) such that
\[
  \dimH(h_\#\mu)\le p.
\]
\end{lemma}

\begin{proof}
Let \(p>0\).  Romney's theorem \cite[Corollary~1.4]{R19} states that
$\dimC\mathcal L^n=0$.  Hence there exist a metric space \(Y\), a
quasisymmetric homeomorphism \(g:\R^n\to Y\), and a Borel set
\(E\subset\R^n\) such that
\[
  \mathcal L^n(\R^n\setminus E)=0
  \quad\text{and}\quad
  \dimH g(E)\le p.
\]

Set \(U=\R^n\setminus E\).  By Lemma~\ref{lem:translation}, choose
\(t_0\in[0,1]^n\) such
that
\[
  \mu\{x\in S:\ \varphi(x)+t_0\in U\}=0.
\]
Define
\[
  h:S\to Y,\qquad h(x)=g(\varphi(x)+t_0).
\]
Since the translation \(z\mapsto z+t_0\) is an isometry, \(h\) is a
quasisymmetric homeomorphism from \(S\) onto \(h(S)\).  The set
\[
  F=\{x\in S:\ \varphi(x)+t_0\notin U\}
\]
has full \(\mu\)-measure in \(S\), and
\[
  h(F)\subset g(E).
\]
Thus \(h(F)\) has full \(h_\#\mu\)-measure. 
Therefore
\[
  \dimH(h_\#\mu)
  \le \dimH h(F)
  \le \dimH g(E)
  \le p.
\]
This proves the lemma.
\end{proof}

We can now combine Assouad's embedding with
Lemma~\ref{lem:euclidean-reduction}.

\begin{proof}[Alternative proof of Corollary~\ref{cor:doubling}]
Let \(S=\spt\mu\).  Since $X$ is doubling, it is separable, and hence
$\mu(X\setminus S)=0$.  We therefore regard $\mu$ as a locally finite Borel
measure on $S$.  Since \(S\) is doubling, Assouad's embedding theorem
\cite[Theorem~12.1]{H01} gives a quasisymmetric embedding
\(\varphi:S\to\R^n\) for some
\(n\ge1\).

Since $S$ has a countable base and \(\mu\) is locally finite, there are Borel
sets \(S_j\subset S\) such that
\[
  S=\bigcup_{j=1}^{\infty}S_j
  \quad\text{and}\quad
  \mu(S_j)<\infty
  \quad\text{for all }j.
\]
Thus \(\mu\) is \(\sigma\)-finite.

By Lemma~\ref{lem:euclidean-reduction}, applied to \((S,d)\), \(\mu\), and
\(\varphi\), for every \(p>0\) there exist a metric space \(Y\) and a quasisymmetric
homeomorphism \(h:S\to h(S)\subset Y\) such that
\(\dimH(h_\#\mu)\le p\).  Since \(p>0\) is
arbitrary, \(\dimC\mu=0\).
\end{proof}

\section{Singular quasisymmetric maps}\label{sec:singular}

This section proves Theorem~\ref{thm:singular}.  We recall the part of
Romney's metric construction needed for the upper bound, use a counting
argument to control the exceptional set, and finish with the sharpness
statement in Proposition~\ref{prop:sharpness}.

\subsection{Romney's construction}
The case $n=1$ of Theorem~\ref{thm:singular} is Tukia's theorem
\cite{T89}, so the rest of this section assumes $n\ge2$.  We first review
Romney's construction \cite[Sections~2--4]{R19} and isolate the diameter
estimate needed in the proof.

Let $Q=[0,1]^n$, equipped initially with the Euclidean metric
$|\cdot|$.  Fix an integer $M>2n$ and a real number $L>1$, to be chosen
later. For $k\ge1$, let
$W_k=\{0,1,\ldots,M^n-1\}^k$, set $W_0=\{\emptyset\}$, and write
$W=\bigcup_{k\ge0}W_k$.  The standard $M$-adic grid
$\mathcal Q_k$ consists
of $M^{nk}$ closed cubes of side length $M^{-k}$ with pairwise disjoint
interiors and union $Q$.  We
label these cubes as $Q_w$, $w\in W_k$, in such a way that
$Q_{(w,i)}\subset Q_w$ for $0\le i<M^n$. Set $Q_{\emptyset}=Q$. We say the
\emph{level} of $Q_w$ is $k$. The vertices of the level-$k$ cubes form the set
\[
\mathcal V_k = \left\{ (i_1 M^{-k}, \ldots, i_n M^{-k}) : 0 \le i_1, \ldots, i_n \le M^k \right\},
\]
and we write $\mathcal V=\bigcup_{k\ge0}\mathcal V_k$.

The \emph{children} of $Q_w\in\mathcal Q_k$ are the cubes in
$\mathcal Q_{k+1}$ contained in $Q_w$.  Divide these children into three
subcollections.  For $E,F\subset Q$, write
$d_{|\cdot|}(E,F)=\inf\{|x-y|:x\in E,\ y\in F\}$ for their Euclidean
distance.  Then
\begin{align*}
\mathcal P_w^1 &= \bigl\{ Q\in \mathcal Q_{k+1} : Q\subset Q_w \text{ and } Q\cap \partial Q_w\neq \varnothing \bigr\}, \\
\mathcal P_w^2 &= \bigl\{ Q\in \mathcal Q_{k+1}\setminus \mathcal P_w^1 : Q\subset Q_w \text{ and } d_{|\cdot|}\bigl(Q, \cup \mathcal P_w^1\bigr) < (n-1)M^{-(k+1)} \bigr\}, \\
\mathcal P_w^3 &= \bigl\{ Q\in \mathcal Q_{k+1} : Q\subset Q_w \text{ and } Q\notin \mathcal P_w^1 \cup  \mathcal P_w^2\bigr\},
\end{align*}
where $\bigcup\mathcal P_w^1$ denotes the union of the cubes in
$\mathcal P_w^1$. These collections form,
respectively, the boundary layer, the intermediate layer, and the central
region of $Q_w$.  Writing $\#\mathcal A$ for the number of elements of a
finite set $\mathcal A$, their cardinalities are
\[
\#\mathcal P_w^1 = M^n - (M-2)^n,\qquad
\#\mathcal P_w^2 = (M-2)^n - (M-2n)^n,\qquad
\#\mathcal P_w^3 = (M-2n)^n.
\]
For $j\in\{1,2,3\}$, set $P_w^j=\bigcup\mathcal P_w^j$.

We next define \emph{conformal weights} $\rho_k:Q\to(0,\infty)$. Set
$\rho_0\equiv1$. Assuming \(\rho_k\) is known,  for $w\in W_k$, we define \(\rho_{k+1}\) on the interiors of the children of \(Q_w\) as follows. Let \(q_w\) denote the center of \(Q_w\). Then
\[
\rho_{k+1}(x) =
\begin{cases}
	\rho_k(q_w), & \text{ if }x\in P_w^1,\\
	(M-2n+1)\,\rho_k(q_w), & \text{ if }x\in P_w^2,\\
	L^{-1}\,\rho_k(q_w), & \text{ if }x\in P_w^3,
\end{cases}
\]
Finally, at each grid-boundary point $x$, define $\rho_{k+1}(x)$ as
the minimum of the constant values assigned on the interiors of the
level-$(k+1)$ cubes containing $x$. This is the
lower-semicontinuous extension of the preceding definition.

For $k\ge0$, the weight $\rho_k$ induces the \emph{length metric}
\[
D_k(x,y) = \inf_\gamma \int_\gamma \rho_k \, ds,
\]
where the infimum is over all curves $\gamma$ of finite Euclidean length
joining $x$ to $y$, and $ds$ is Euclidean arclength. Romney proves that, for each
$x,y\in\mathcal V$, the sequence $(D_k(x,y))$ is eventually
nondecreasing and bounded. Thus
\[
D(x,y)=\lim_{k\to\infty}D_k(x,y),\qquad x,y\in\mathcal V,
\]
defines a metric on $\mathcal V$ that extends by continuity to a metric on $Q$;
see \cite[Section~4]{R19}. Moreover, the identity map
\[
f\colon Q \to X, \qquad X=(Q,D)
\]
is quasisymmetric by \cite[Proposition~4.1]{R19}.

For $w\in W_k$, write $r(w)=\rho_k(q_w)$. For $0\le i\le k$,
let $w^{(i)}\in W_i$ be the prefix of $w$ of length $i$. Note that $Q_w\subset Q_{w^{(i)}}$ for each $i$. Define
\[
c(w)=\#\bigl\{i\in\{1,\ldots,k\}:Q_{w^{(i)}}
\in\mathcal P_{w^{(i-1)}}^3\bigr\}.
\]
Thus $c(w)$ is the number of visits to the central region during the
first $k$ steps. Since each noncentral step multiplies the weight by at
most $A:=M-2n+1$, while each central step multiplies it by $L^{-1}$,
\begin{equation}\label{eq:romney-weight}
r(w)\leq A^k L^{-c(w)}.
\end{equation}
The diameter estimate used in Subsection~\ref{subsec:singular-proof} is
an immediate consequence of \cite[Lemma~3.3]{R19}.

\begin{lemma}\label{lem:romney-diameter}
There is a constant $C_0=C_0(n,M,L)<\infty$ such that, for every
$k\ge0$ and every $Q_w\in\mathcal Q_k$,
\[
  \diam_D f(Q_w)
  \le
  C_0 r(w)M^{-k}\leq C_0 A^k L^{-c(w)} M^{-k},
\]
where $A=M-2n+1$.
\end{lemma}
\begin{proof}
Fix a vertex $z$ of $Q_w$. Given $x\in Q_w$, choose
$x_j\in\mathcal V\cap Q_w$ with $x_j\to x$ in the Euclidean metric.
The continuous extension of $D$ and \cite[Lemma~3.3]{R19} give
\[
\begin{aligned}
D(x,z)
 &=\lim_{j\to\infty}D(x_j,z)
  =\lim_{j\to\infty}\lim_{m\to\infty}D_m(x_j,z)\\
 &\le C_1r(w)\lim_{j\to\infty}|x_j-z|
  =C_1r(w)|x-z|
 \le C_1\sqrt n\,r(w)M^{-k},
\end{aligned}
\]
where $C_1=C_1(n,M,L)$. Applying this estimate to two points of $Q_w$
and using the triangle inequality yields
$\diam_D f(Q_w)\le2C_1\sqrt n\,r(w)M^{-k}$. The result now follows from
\eqref{eq:romney-weight}, with $C_0=2C_1\sqrt n$.
\end{proof}

\subsection{Proof of Theorem~\ref{thm:singular}}
\label{subsec:singular-proof}
The preceding lemma gives the metric estimate needed here.  We combine it
with a count of cubes that visit the central region too infrequently.
Fix $0<\theta<1/2$.  A cube $Q_w\in\mathcal Q_k$ is called
\emph{$\theta$-good} if
\[
  c(w)\ge \theta k.
\]
Otherwise it is called \emph{$\theta$-bad}.  Let $G_k$ be the union of all
$\theta$-good cubes in $\mathcal Q_k$, and define
\begin{equation}\label{eq:Etheta}
  E_\theta
  :=
  \limsup_{k\to\infty}G_k
  =
  \bigcap_{N=1}^{\infty}\bigcup_{k\ge N}G_k.
\end{equation}
Thus $E_\theta$ is the Borel set of points that belong to $\theta$-good cubes
at infinitely many levels.  We first count the bad cubes, then estimate the
Euclidean dimension of the exceptional set and the $D$-dimension of
$f(E_\theta)$.

\begin{lemma}\label{lem:bad-cube-count}
There is a constant $C_n>1$, depending only on $n$, such that, for all
sufficiently large $M$, the number $B_k$ of $\theta$-bad cubes in
$\mathcal Q_k$ satisfies
\[
  B_k
  \le
  (k+1)\exp(kH(\theta))
  (C_n M^{n-1})^k
  (C_n M)^{\theta k},
\]
where
\[
  H(\theta)
  =
  -\theta\log\theta-(1-\theta)\log(1-\theta)
\]
is the \emph{entropy function}.
\end{lemma}

\begin{proof}
Set
\[
a=M^n-(M-2n)^n
\quad\text{and}\quad
b=(M-2n)^n.
\]
The numbers of noncentral and central choices at each step are $a$ and
$b$, respectively. Consequently,
\[
  B_k
  =
  \sum_{j<\theta k}
  \binom{k}{j}
  b^j a^{k-j}=\sum_{j<\theta k}
  \binom{k}{j}
  \left(\frac{b}{a}\right)^j a^k.
\]

For all sufficiently large $M$, a constant $C_n>1$ depending only on
$n$ can be chosen so that
\[
C_n^{-1} M^{n-1}\leq a\leq C_n M^{n-1}, \qquad \frac{b}{a}\leq C_n M.
\]
The standard entropy estimate
\[
\binom{k}{j}\le \exp\bigl(kH(j/k)\bigr),\qquad 0\le j\le k,
\]
holds with the convention $0\log0=0$. Since $H$ is increasing on
$[0,1/2]$ and $0<\theta<1/2$,
\[
  \sum_{j< \theta k}\binom{k}{j}
  \le (k+1)\exp(k H(\theta)).
\]
For $j<\theta k$, we also have
$(b/a)^j\le(C_nM)^{\theta k}$. Combining these estimates proves the
lemma.
\end{proof}

\begin{lemma}\label{lem:bad-set-dimension}
For every $p>0$, there is a $\theta\in(0,1/2)$ such that, for all
sufficiently large $M$,
\[
  \dimH(Q\setminus E_\theta)\le n-1+p.
\]
\end{lemma}

\begin{proof}
For $N\ge1$, set
\[
F_N:=\bigcap_{k\ge N}(Q\setminus G_k).
\]
Then $Q\setminus E_\theta=\bigcup_{N\ge1}F_N$. For every $k\ge N$,
the set $F_N$ is covered by the $\theta$-bad cubes in $\mathcal Q_k$,
each of which has Euclidean diameter at most $\sqrt n\,M^{-k}$.

Let
\[
  s>
  n-1+\theta+
  \frac{H(\theta)+(\theta+1)\log C_n}{\log M}.
\]
By Lemma~\ref{lem:bad-cube-count}, the covering sum in the definition of
$\mathcal H^s(F_N)$, using cubes of diameter at most $\sqrt nM^{-k}$, is
bounded by
\[
  B_k(\sqrt n\,M^{-k})^s
  \le
  (k+1)n^{s/2}
  \left[
    \exp(H(\theta))C_n^{\theta+1}
    M^{n-1+\theta-s}
  \right]^k.
\]
The bracket is strictly less than $1$. Letting $k\to\infty$ gives
$\mathcal H^s(F_N)=0$. Since $Q\setminus E_\theta$ is the countable union
of the sets $F_N$,
\[
  \dimH(Q\setminus E_\theta)
  \le
  n-1+\theta+
  \frac{H(\theta)+(\theta+1)\log C_n}{\log M}.
\]
First choose $\theta>0$ sufficiently small and then choose $M$
sufficiently large so that the right-hand side is at most
$n-1+p$.
\end{proof}

\begin{lemma}\label{lem:good-image-dimension}
Fix $p>0$, $\theta\in(0,1/2)$, and $M>2n$. If $L$ is sufficiently
large, then
\[
  \dimH f(E_\theta)\le p.
\]
\end{lemma}

\begin{proof}
The set $G_k$ is the union of at most $M^{nk}$ cubes $Q_w$ satisfying
$c(w)\ge\theta k$. By Lemma~\ref{lem:romney-diameter}, each such cube satisfies
\[
  \diam_D f(Q_w)
  \le
  C_0 A^kL^{-\theta k}M^{-k}
  =
  C_0\left(\frac{A}{M}\right)^kL^{-\theta k},
\]
where $A=M-2n+1$.
For every $N\ge1$,
\[
  E_\theta
  \subset
  \bigcup_{k\ge N}G_k,
\]
so these images form a cover of $f(E_\theta)$. The sum of the
$p$-powers of their diameters is at most
\[
  \sum_{k\ge N}
  M^{nk}
  \left[
    C_0\left(\frac{A}{M}\right)^kL^{-\theta k}
  \right]^p.
\]
This equals
\[
  C_0^p
  \sum_{k\ge N}
  \left[
    M^n\left(\frac{A}{M}\right)^p L^{-p\theta}
  \right]^k.
\]
Choose $L$ so large that
\[
  \lambda
  :=
  M^n\left(\frac{A}{M}\right)^p L^{-p\theta}
  <1.
\]
Although $C_0$ may depend on $L$, it is independent of $k$ and appears only
as the fixed prefactor $C_0^p$ in the preceding series.
Since $\lambda<1$, the diameters in this cover tend uniformly to zero
as $k\to\infty$, and the displayed sum tends to zero as $N\to\infty$.
Hence
\[
  \mathcal{H}^p(f(E_\theta))=0,
\]
and therefore
\[
  \dimH f(E_\theta)\le p.
\]
\end{proof}

\begin{proof}[Proof of Theorem~\ref{thm:singular}]
The case $n=1$ follows from Tukia's theorem \cite{T89}, so assume
$n\ge2$. Romney's construction with parameters $M,L$ gives a metric
space
\[
  X=(Q,D)
\]
and a quasisymmetric homeomorphism
\[
  f:(Q,|\cdot|)\to X.
\]
For $p>0$, choose $\theta\in(0,1/2)$ and then $M$ so that
Lemma~\ref{lem:bad-set-dimension} gives
\[
  \dimH(Q\setminus E_\theta)\le n-1+p.
\]
Next choose $L$ so that Lemma~\ref{lem:good-image-dimension} gives
\[
  \dimH f(E_\theta)\le p.
\]
The set $E=E_\theta$ is Borel by \eqref{eq:Etheta}, and it has the required
properties.
\end{proof}

We finish by showing that, in higher dimensions, one cannot obtain the precise
analogue of Tukia's conclusion, with both $[0,1]^n\setminus E$ and
$f(E)$ of arbitrarily small Hausdorff dimension.

\begin{proposition}\label{prop:sharpness}
Let $n\ge2$, let $f:[0,1]^n\to X$ be a homeomorphism onto a metric
space $X$, and let $E\subset[0,1]^n$. If
\[
  \dimH f(E)<1,
\]
then
\[
  \dimH([0,1]^n\setminus E)\ge n-1.
\]
\end{proposition}

\begin{proof}
Still write $Q=[0,1]^n$. Suppose, toward a contradiction, that
\[
  \dimH(Q\setminus E)<n-1.
\]
Since \(f\) is a homeomorphism, $X$ is a separable metric space satisfying
\[
  X=f(E)\cup f(Q\setminus E).
\]
For a separable metric space $Z$, its \emph{small inductive dimension}
$\operatorname{ind}Z$ agrees with its \emph{covering dimension} and satisfies
$\operatorname{ind}Z\le\dimH Z$; see \cite[Theorem~6.3.10]{E08} and
\cite[Theorem~4.1.5]{E78}.  Hence
\[
  \operatorname{ind} f(E)<1
  \quad\text{and}\quad
  \operatorname{ind} f(Q\setminus E)
  =
  \operatorname{ind} (Q\setminus E)
  \le
  \dimH(Q\setminus E)
  <
  n-1.
\]
Here the equality uses the topological invariance of $\operatorname{ind}$
\cite[Theorem~3.4.2]{E08}.  Since $\operatorname{ind}$ is integer-valued on
separable metric spaces, and the contradiction assumption forces $E$ to be
nonempty, we obtain
\[
  \operatorname{ind} f(E)=0,
  \qquad
  \operatorname{ind} f(Q\setminus E)\le n-2.
\]
The union theorem \cite[Theorem~3.4.11]{E08} now gives
\[
  \operatorname{ind} X
  \le
  \operatorname{ind} f(E)+\operatorname{ind} f(Q\setminus E)+1
  \le
  0+(n-2)+1
  =
  n-1.
\]
But \(X\) is homeomorphic to \(Q=[0,1]^n\), so
\[
  \operatorname{ind} X=n,
\]
a contradiction.  Hence \(\dimH(Q\setminus E)\ge n-1\).
\end{proof}

\subsection*{Conflicts of interest}
The authors declare that they have no conflicts of interest.

\subsection*{Data availability statement}
No datasets were generated or analyzed in this study.

\end{document}